\documentclass[nolayout]{article}
\usepackage[utf8]{inputenc}
\usepackage[english]{babel}
\usepackage{authblk}
\usepackage{amssymb,amsmath,amsthm}
\usepackage{times,dsfont}
\usepackage{fancyhdr}
\usepackage{geometry}
\title{Identifiability and consistent estimation of nonparametric translation hidden Markov models with general state space}
\date{} 

\author[*]{\'Elisabeth Gassiat}
\author[$\dag$]{Sylvain Le Corff}
\author[*]{Luc Leh\'ericy}

\affil[*]{{\small Laboratoire de Math\'ematiques d'Orsay, Univ. Paris-Sud, CNRS, Universit\'e Paris-Saclay, 91405 Orsay, France.}}
\affil[$\dag$]{{\small  Samovar, T\'el\'ecom SudParis,  D\'epartement CITI, TIPIC, Institut Polytechnique de Paris, France.}}

\newcommand{\M}{\mathcal{M}}

\newcommand{\N}{\mathbb{N}}
\newcommand{\R}{\mathbb{R}}
\newcommand{\C}{\mathbb{C}}

\newcommand{\E}{\mathbb{E}}

\newcommand{\po}{\mathbb{P}}
\renewcommand{\tilde}{\widetilde}
\renewcommand{\hat}{\widehat}

\newcommand\pen{\mathrm{pen}}

\DeclareMathOperator*{\argmax}{arg\,max}

\newcommand\noisedist{P}

\newcommand\noisedistbis{\widetilde{P}}

\newcommand\eqsp{\,}
\newcommand\rmd{\mathrm{d}}
\newcommand\rme{\mathrm{e}}

\newcommand{\Acal}{{\mathcal{A}}}
\newcommand{\Ccal}{{\mathcal{C}}}
\newcommand{\Fcal}{{\mathcal{F}}}
\newcommand{\Kbf}{{\mathbf{K}}}
\newcommand{\Lbf}{{\mathbf{L}}}
\newcommand{\Mcal}{{\mathcal{M}}}

\newcommand{\Pcal}{{\mathcal{P}}}

\newcommand{\Rcal}{{\mathcal{R}}}
\newcommand{\Scal}{{\mathcal{S}}}

\newcommand{\Zbb}{{\mathbb{Z}}}
\newcommand{\Bcal}{{\mathcal{B}}}
\newcommand{\majZ}{\Xi}
\newcommand{\Xfrak}{{\mathfrak{X}}}
\newcommand{\Sbf}{{\mathbf{S}}}
\newcommand{\Zsf}{{\mathsf{Z}}}
\newcommand{\Id}{{\text{Id}}}
\newcommand{\Xcal}{{\mathcal{X}}}
\newcommand{\Nbb}{\mathbb{N}}
\newcommand{\Pfrak}{{\mathfrak{P}}}
\newcommand{\Tbf}{{\mathbf{T}}}

\newcommand{\one}{{\mathbf{1}}}
\newcommand{\Supp}{{\text{Supp}}}

\renewcommand{\geq}{\geqslant}
\renewcommand{\leq}{\leqslant}

\usepackage{booktabs} 
\usepackage{colortbl} 
\usepackage{xcolor} 
\usepackage{xfrac}
\usepackage{enumerate}

\newtheorem{theo}{Theorem}
\newtheorem{prop}{Proposition}
\newtheorem{cor}{Corollary}
\newtheorem{lem}{Lemma}

\newtheorem*{ex}{Example}

\newcounter{hypH}
\newenvironment{hypH}{\refstepcounter{hypH}\begin{itemize}
\item[\textbf{H\arabic{hypH}}]}{\end{itemize}}

\begin{document}

\maketitle

\begin{abstract}
This paper considers hidden Markov models where the observations are given as the sum of a latent state which lies in a general state space and some independent noise with unknown distribution. It is shown that these fully nonparametric translation models are identifiable with respect to both the distribution of the latent variables and the distribution of the noise, under mostly a light tail assumption on the latent variables. Two nonparametric estimation methods are proposed and we prove that the corresponding estimators are consistent for the weak convergence topology. These results are illustrated with numerical experiments.
\end{abstract}


\section{Introduction}
This paper considers nonparametric translation hidden Markov models where, for all $i = 1, \dots, n$, the observation $Y_i$ is 
\begin{equation}
\label{eq:model:deconvolution}
Y_i = X_i + \varepsilon_i\eqsp,
\end{equation}
where $n\geqslant 1$ is the number of observations, $(X_i)_{i = 1, \dots, n}$ is a $d$ dimensional hidden stationary Markov chain  and $(\varepsilon_i)_{i = 1, \dots, n}$ are independent and identically distributed random variables independent of $(X_i)_{i = 1, \dots, n}$. Both the distributions of the latent variables and of the noise $\varepsilon_1$ are unknown. The first objective of this paper is to prove that the law of the hidden states may be recovered using only the observations $(Y_i)_{i = 1, \dots, n}$  when no assumption is made on the noise distribution and with only a weak nonparametric assumption on the distribution of the hidden Markov chain.  In addition, consistent estimation procedures based either on a least squares or on a maximum likelihood approach are proposed. This work provides the first contribution to establish identifiability results in a fully nonparametric setting for hidden Markov models with general state space.

The use of latent data models is ubiquitous in time series analysis across a wide range of applied science
and engineering domains such as  signal processing \cite{crouse:nowak:baraniuk:1998}, genomics \cite{yau:papas:roberts:holmes:11,wang:lebarbier:aubert:robin:2017},  target tracking \cite{sarkka:vehtari:lampinen:2007}, enhancement and segmentation of speech and audio signals \cite{rabiner:89}, see also \cite{sarkka:2013,douc:moulines:stoffer:2014, zucchini:macdonald:langrock:2016} and the numerous references therein. The specific setting of translation hidden Markov models described by \eqref{eq:model:deconvolution} is commonly used in statistical signal processing, such as for nonlinear phase estimation, where the problem appears in many applications: detection of phase synchronization, estimation of instantaneous frequencies or in neuroscience, see \cite{dahlhaus:kiss:neddermeyer:2018}, \cite{fell:axmacher:2011} and the references therein. In these applications, the latent signal is modeled as $X_{i} = g(Z_{i})$, for some sequence $(Z_{i})_{i\geq 1}$ of relevant hidden variables and function $g:\R^{\ell} \rightarrow \R^{d}$. In \cite{dumont:lecorff:2017b}, such models are used to detect oscillation patterns in human electrocardiogam recordings and to estimate a noisy Rossler attractor. In such a case, model (\ref{eq:model:deconvolution}) is a
nonparametric hidden regression model given by
\begin{equation}
\label{eq:model:regresshidden}
Y_i = g\left(Z_i \right)+ \varepsilon_i\eqsp, \;i\geq 1\eqsp.
\end{equation}

Although parametric hidden Markov models have been widely studied and are appealing for a wide range of applications, parametric inference procedures may lead to poor results in real data and high dimensional learning problems. This explains the recent  keen interest for nonparametric latent data models which have been introduced in many disciplines such as climate state identification \cite{lambert:whiting:metcalfe:03,touron:2018}, genomics \cite{yau:papas:roberts:holmes:11}, statistical modelling of animal movement \cite{langrock2015nonparametric} or biology \cite{volant:berard:2014}. \cite{levine2011maximum} introduce an iterative algorithm with similar monotonicity property as the Expectation Maximization algorithm to estimate a nonparametric finite mixture of multivariate components with applications to simulated data and to the water-level dataset (see the mixtools package). 
 In \cite{langrock2017markov},  Markov-switching generalized additive models where the function $g$ and the noise distribution in \eqref{eq:model:regresshidden} depend on a hidden label are used to describe signals with complex dynamic patterns. The authors of this paper introduced an efficient  nonparametric estimation method of the unknown functions of the hidden signal.
The spline-based nonparametric estimation of these functionals is applied to advertising data and to Spanish energy price data, see also \cite{langrock2015nonparametric} for an application of nonparametric regression estimation with P-Splines to the vertical speed of diving beaked whales.

For finite state space hidden Markov models, such nonparametric modeling  has been recently validated by theoretical identifiability results and the analysis of estimation procedures with provable guarantees, 
see \cite{MR3439359}, \cite{MR3509896}, \cite{MR3543517}, \cite{LucJMLR}. In this setting, the parameters to be estimated are  the transition matrix of the hidden chain and the emission densities. 
See also \cite{gassiat:rousseau:2016} and \cite{akakpo_translationhmm} for translation hidden Markov models with finite state space.
While certainly of interest, the finite state space setting may be too restrictive for many applications. 

The inverse problem in (\ref{eq:model:deconvolution}) is also known as the deconvolution problem. There is a wide range of literature on density deconvolution when the distribution of the noise $ \varepsilon_i$ is assumed to be known and  the random variables $(X_i,\varepsilon_i)_{i = 1, \dots, n}$ are assumed to be independent and identically distributed, see 
\cite{MR1033106}, 
\cite{MR1047309}, \cite{MR1054861},
for some early nonparametric deconvolution methods, \cite{MR997599} and \cite{MR1126324} for minimax rates, see also \cite{MR3314482}
and references therein for a recent work. However, when the distribution of  the noise is unknown and the observations are independent, model (\ref{eq:model:deconvolution}) can not be identified in full generality.

In this paper, we establish the identifiability of the fully nonparametric hidden translation model under the weak assumption that the Laplace transform of the latent variable has an exponential growth smaller than $2$ and some assumption on the distribution of two consecutive hidden states which is roughly a dependency assumption, see  Theorem \ref{theoident1}. In the case of real valued hidden Markov models, identifiability is extended to latent variables having Laplace transform with exponential growth smaller than $3$, see Theorem \ref{theoident2}. Two different methods are proposed to recover the distribution of the latent variables:  a least squares method arising naturally from the identifiability proof and a classical maximum likelihood method using discrete probability measures as approximation of all probability measures. Both estimators are proved to be consistent for the weak convergence topology, see Theorem \ref{theoconsi1} and Theorem \ref{theoconsi2}. 
The most surprising result is that the identifiability of the signal distribution only requires an assumption on the tail of its distribution and a dependency assumption, and does not require any assumption on the unknown distribution of the noise. This has to be compared to works such as \cite{Wilhelm2015},  \cite{MR2970334} or \cite{MR2374986} in which some conditions require several operators to be injective and some variables  to have densities.
It is also important to note that Theorem \ref{theoident1} encompasses the case of dependent observations in which the hidden signal is not necessarily a Markov chain, which can be the case  for the nonparametric hidden regression model (\ref{eq:model:regresshidden}). In that sense, our work extends the identification results of \cite{dumont:lecorff:2017,dumont:lecorff:2017b} to the cases where the distribution of the additive noise is unknown. 
Such a general result provides also the first theoretical guarantees for the identification of nonparametric latent data models which have been applied in various frameworks such as in \cite{langrock2017markov} or \cite{langrock2015nonparametric}.

The paper is organized as follows. Section~\ref{sec:identifiability:general} displays the general identifiability results. The consistency of the least squares approach and that of  the maximum likelihood estimation procedures are given in Section \ref{sec:estimation}. These results are supported by simulations  in Section \ref{sec:simu}.  Section~\ref{sec:discussion} provides a synthesis of the results obtained in the paper and discusses some opportunities for further research. In particular, it points out an important, yet very challenging, unsolved problem in the setting of this paper: obtainning convergence rates in deconvolution problems where absolutely no information about the noise is available.

\section{Identifiability Theorems}
\label{sec:identifiability:general}
Consider a sequence of random variables  $(Y_{i})_{i\geq 1}$ taking values in $\R^{d}$ and satisfying model (\ref{eq:model:deconvolution}) in which the hidden Markov chain $(X_{i})_{i\geq 1}$ is stationary.
In the following, $\R^{d}$ is endowed with its Borel sigma-field $\Bcal(\R^d)$.
For each transition kernel $K: \R^d\times \mathcal{B}(\R^d) \to [0,1]$ with a unique stationary  distribution $\mu_K$, define the measure $R_{K}$ on $\R^{2d}$ as follows. For all $E\in \mathcal{B}(\R^{2d})$,
$
R_{K}(E) = \int \mu_K(\rmd x) K(x,\rmd y)\mathds{1}_E(x,y)\eqsp.
$
For any probability distribution $P$ on $\R^d$, denote by 
$\po_{K,P}$ the distribution of the sequence $(Y_{i})_{i\geq 1}$ when the stationary Markov chain  $(X_{i})_{i\geq 1}$ has transition $K$ and $\varepsilon_{1}$ has distribution $P$. 
For any $\rho >0$, let $\M_{\rho}$ be the set of finite measures $\mu$ on $\R^{d}$ such that there exist $A,B>0$ satisfying, for all $u\in \R^{d}$,
$\int\exp \left(u^{T}x\right) \rmd \mu (x)
\leq A \exp \left( B \|u\|^{\rho}\right)$,
where for a vector $u$ in a Euclidian space, $\|u\|$ denotes its Euclidian norm and $u^{T}$ denotes its transpose vector.
If $K$ is such that $\mu_{K} \in {\M_{\rho}}$ for some $\rho$, then the function $\Phi_{R_{K}}$ defined  for  $(z_{1},z_{2})\in\C^{d}\times \C^{d}$ by
$\Phi_{R_{K}}(z_{1},z_{2})=
\int  \exp \left(z_{1}^{T}x_{1} +z_{2}^{T}x_{2}\right) \rmd R(x_{1}, x_{2})
$
is well defined over $\C^{d}\times \C^{d}$. 
Consider the following assumption.
\begin{hypH}
\label{assum::}
For any $z_{0}\in \C^{d}$, 
$
z \mapsto \Phi_{R_{K}}(z_{0},z)
$
is not the null function or
$
z \mapsto \Phi_{R_{K}}(z,z_{0})
$
is not the null function.
\end{hypH}
An alternative equivalent formulation of Assumption H\ref{assum::} is the following: 
$
{\text{for any }}z_{0}\in \C^{d}, \;
\mathbb{E}[  \rme^{z_{0}^{T}X_{2}} \vert X_{1} ] \neq 0
{\text{ or }}
\mathbb{E}[\rme^{z_{0}^{T}X_{1}} \vert X_{2}] \neq 0$. Throughout this paper, the assertion \emph{$R_K = R_{\tilde K}$  and $\noisedist = \noisedistbis$ up to translation} means that there exists $m\in  \R^{d}$ such that if  $(X_1,X_2)$ has distribution $R_K$ and $(\varepsilon_{1},\varepsilon_2)$ has distribution $\noisedist\otimes \noisedist$, then $(X_{1}-m,X_2-m)$ has distribution $ R_{\tilde K}$ and  $(\varepsilon_{1}+m,\varepsilon_{2}+m)$ has distribution $\noisedistbis\otimes \noisedistbis
$. The following theorems state that the distribution of the observations allows to recover the kernel of the hidden Markov chain and the distribution of the noise up to translation.

\begin{theo}
\label{theoident1}
Assume that  $K$ (resp. $\widetilde K$) is a transition kernel on $\R^d\times \mathcal{B}(\R^d)$ with a unique stationary distribution $\mu_K$ (resp. $\mu_{\widetilde K}$) and that $R_{K}$ and $R_{\widetilde K}$ satisfy assumption H\ref{assum::}.
Assume also that there exists $\rho <2$ such that $\mu_{K}\in \M_{\rho}$ and $\mu_{\widetilde K}\in \M_{\rho}$. Then, $\po_{{K},P}=\po_{{\widetilde K},\widetilde P}$ implies that $R_{K} =R_{\widetilde K}$ and $\noisedist = \noisedistbis$ up to translation.
\end{theo}
In the case of real valued random variables, identifiability holds for a   class of transition kernels including Gaussian Markov chains.

\begin{theo}[case $d=1$]
\label{theoident2}
Assume that $K$ (resp. $\widetilde K$) is a transition kernel on $\R\times \mathcal{B}(\R)$ with a unique stationary distribution $\mu_K$ (resp. $\mu_{\widetilde K}$) and with a density with respect to the Lebesgue measure. Assume that there exists $\rho <3$ such that $\mu_{K}\in \M_{\rho}$ and $\mu_{\widetilde K}\in \M_{\rho}$. Assume that $R_{K}$ and $R_{\widetilde K}$ satisfy assumption H\ref{assum::}. Assume moreover that if the stationary Markov chain with transition kernel $K$ (resp. $\widetilde K$) is Gaussian, it is not a sequence of independent and identically distributed variables. Then, $\po_{K,\noisedist}=\po_{{\widetilde K},\noisedistbis}$ implies that $R_{K} = R_{\widetilde K}$ and $\noisedist = \noisedistbis$ up to translation.
\end{theo}
One way to fix the ``up to translation" indeterminacy when the noise has a first order moment is to assume that $\mathbb{E}[\varepsilon_{1}]=0$.
Detailed proofs of Theorems~\ref{theoident1} and~\ref{theoident2} can be found in Appendix~\ref{sec:proof:ident1}.

\bigskip \noindent
\emph{Comments on the assumptions.}
\begin{enumerate}[i)]
\item The assumption that $\mu_{K}\in \M_{\rho}$ is an assumption on the tails of the distribution $\mu_{K}$. If $\mu_{K}$ is compactly supported, then $\mu_{K}\in \M_{1}$, and if a probability distribution is in $\M_{\rho}$ for some $\rho$, then $\rho \geq 1$ except in case it is a Dirac mass at point $0$. The assumption $\rho <2$ means that $\mu_{K}$ is required to have tails lighter than that of Gaussian distributions. 

\item The most striking result is that there is no assumption at all on the distribution of the noise, it could have any distribution (including the possibility of a deterministic noise). In particular, there is no assumption on the set where its  characteristic function vanishes. In addition, there is no density or singularity assumption on the distribution of the hidden signal. The hidden sequence may have atomic distributions, continuous distributions, and no specific knowledge about this is required. The only assumptions are on the tail of the signal distribution and assumption H\ref{assum::} which, as discussed below, is a dependency assumption. In contrast, in previous works such as \cite{Wilhelm2015},  \cite{MR2970334} or \cite{MR2374986}, part of the observations and hidden variables  are assumed to have densities (with boundedness or tail assumptions), and some assumptions require invertibility of operators which in the context of this paper translate to non vanishing of the characteristic function of the noise. The completeness assumption ID4 in \cite{Wilhelm2015} implies H\ref{assum::} if the hidden variables are bounded. 

\item {\bf Assumption H\ref{assum::}}. Hadamard's factorization theorem states that entire functions are completely determined by their set of zeros up to a multiplicative indeterminacy which is the exponential of a polynomial with degree at most the exponential growth of the function (here $\rho$). If  $\mu_{K}\in \M_{\rho}$ for some $\rho <2$, then a consequence of Hadamard's factorization theorem (arguing variable by variable)  is that $\Phi_{R_K}\left(\cdot,0\right)$ has no zeros if and only if $\mu_{K}\in \M_{\rho}$ is a dirac mass. 
A simple example in which Assumption H\ref{assum::} holds is when the Markov chain is a an autoregressive process, that is there exists a function $h$ and a sequence of i.i.d. centered random variables $\eta_{i}$, $i\geq 1$, such that for all integer $i$, $X_{i+1}=h(X_{i})+\eta_{i}$. Indeed in this case, for all $z_{1}\in\C^{d}$ and $z_{2}\in\C^{d}$,
$\Phi_{R_K}\left(z_{1},z_{2}\right)=\E \exp (z_{1} X_{1}+ z_{2} h(X_{1})) \E \exp z_{2} \eta_{1}$, so that for any $z_{1}$, there exists $z_{2}$ such that $\Phi_{R_K}\left(z_{1},z_{2}\right)\neq 0$ and Assumption H\ref{assum::} holds.

Now, if the variables $X_{i}$, $i\geq 1$, are independent, then for all $z_{1}\in\C^{d}$ and $z_{2}\in\C^{d}$,
$\Phi_{R_K}\left(z_{1},z_{2}\right)=\Phi_{R_K}\left(z_{1},0\right)\Phi_{R_K}\left(0,z_{2}\right)$. But if  $X_{i}$ is not deterministic, then the function $\Phi_{R_K}\left(\cdot,0\right)=\Phi_{R_K}\left(0,\cdot\right)$ has zeros, and Assumption H\ref{assum::} does not hold.
In other words, Assumption H\ref{assum::} implies that the variables $X_{i}$, $i\geq 1$ are not independent except if they are deterministic.
When the hidden variables have a finite support set of cardinality $2$, Assumption H\ref{assum::} is even equivalent to the fact that $X_{1}$ and $X_{2}$ are not independent.
\end{enumerate}

\bigskip \noindent
\emph{Outline of the proofs.}
The strategy is to write the characteristic functions of the distribution of two consecutive observations under $\po_{K,\noisedist}$ and $\po_{{\widetilde K},\noisedistbis}$, and to derive an equality involving the characteristic functions of $R_{K}$ and  $R_{\widetilde K}$  in a  neighborhood of the origin using the fact that, in such a neighborhood, the characteristic function of the noise distribution is nonzero both under $\noisedist$ and $\noisedistbis$. Then, the assumption that $\mu_{K}\in \M_{\rho}$ for some $\rho$ is used in two main steps.
\begin{enumerate}[i)]
\item The first step is to extend the equality that holds between characteristic functions in a neighborhood of the origin into an equality on $\C^{d}\times \C^{d}$, that is for any $z_{1}\in\C^{d}$ and $z_{2}\in\C^{d}$,
\begin{equation}
\label{analyticmain}
\Phi_{R_K}\left(z_{1},z_{2}\right)\Phi_{R_{\tilde K}}\left(z_{1},0\right)\Phi_{R_{\tilde K}}\left(z_{2},0\right)=\Phi_{R_{\tilde K}}\left(z_{1},z_{2}\right)\Phi_{R_K}\left(z_{1},0\right)\Phi_{R_K}\left(z_{2},0\right)
\eqsp.
\end{equation}
This equation is also the starting point of \cite{gassiat:rousseau:2016}, but dealing with continuous multidimensional state spaces requires further developments. In the proof of Theorem~\ref{theoident1} and Theorem~\ref{theoident2}, this is possible due to the fact that the functions $\Phi_{R_{K}}$ and $\Phi_{R_{\widetilde K}}$ are multivariate analytic functions.
Indeed, it is possible to replace, in the integral defining $\Phi_{R_{K}}$ and $\Phi_{R_{\widetilde K}}$, the exponential term  by its series expansion and interchange sums and integration using dominated convergence since Laplace transforms under $R_{K}$ with  $R_{\widetilde K}$ are finite everywhere, and the remaining in the series may be locally uniformly upper bounded. 
Then, using assumption H\ref{assum::}, it is possible to prove that variable by variable,  $\Phi_{R_{K}}$ and $\Phi_{R_{\widetilde K}}$ have the same sets of zeros in $\C$ (with multiplicity) when all other variables are fixed. 
\item The second step is to use (again variable by variable) Hadamard's factorization theorem  for entire functions, see \cite[Chapter~5, Theorem~5.1]{Stein:complex}, to prove that $\Phi_{R_{K}}$ and $\Phi_{R_{\widetilde K}}$ are equal up to the exponential of a polynomial of degree at most the integer part of $\rho$. This is where the constraint on $\rho$ is used. Indeed, if $\rho <2$, its integer part is $1$. We prove that $\Phi_{R_{K}}$ and $\Phi_{R_{\widetilde K}}$ are equal up to the exponential of a polynomial of degree at most $1$ in each variable (the case $d>1$ requires a careful analysis), from which we deduce that $R_{K} = R_{\widetilde K}$ and $\noisedist = \noisedistbis$ up to translation. In the case where only $\rho <3$ is required, the conclusion is that $\Phi_{R_{K}}$ and $\Phi_{R_{\widetilde K}}$ are equal up to the exponential of a polynomial of degree at most $2$. In this case, we were able to extend the result only for real valued observations, by proving the following lemma which is used to conclude that the polynomial has to be of degree at most $1$.
\begin{lem}
\label{lem:gaussMarkov}
Assume that $(X_{i})_{i\geq 1}$ is a stationary real valued Markov chain with transition kernel having a density with respect to the Lebesgue measure. Assume that $(\eta_{i})_{i\geq 1}$ is a sequence of  independent and identically distributed real valued Gaussian random variables with positive variance and independent of  $(X_{i})_{i\geq 1}$.
If $(X_{i}+\eta_{i})_{i\geq 1}$ is Markov chain, then $(X_{i})_{i\geq 1}$  is an independent and identically distributed sequence.
\end{lem}
\end{enumerate}
The proof of Theorem \ref{theoident2} uses the fact that the hidden variables form a Markov chain by using Lemma \ref{lem:gaussMarkov} where the Markovian property is the starting point of the proof. This is not the case for Theorem \ref{theoident1} in which only the dependency assumption H\ref{assum::}  is used. Thus, Theorem \ref{theoident1} can be extended to other dependent observations in which the hidden signal is not necessarily a Markov chain. Consider now model (\ref{eq:model:regresshidden}). Applying Theorem  \ref{theoident1} (with the assumptions on $R$ as defined below instead of on the kernel of the Markov chain) yields the following corollary.
\begin{cor}
\label{cor:reg}
Assume that $(Z_{i})_{i\geq 1}$ is a sequence of stationary random variables such that the distribution of $g(Z_{1})$ is in ${\cal M}_{\rho}$ for some $\rho <2$, and such that the distribution $R$ of 
$(g(Z_{1}), g(Z_{2}))$ satisfies Assumption H\ref{assum::}. Then the application that associates $R$ (in $\M_{\rho}$ for some $\rho <2$ and satisfying H\ref{assum::}) and $P$ (the distribution of the noise) to the distribution of $(Y_{1},Y_{2})$ is one-to-one up to translation.
\end{cor}
%
The function $g$ can be known or unknown. If $g$ is unknown but may be recovered from the knowledge of the distribution of $(g(Z_{1}),g(Z_{2}))$, then Corollary~\ref{cor:reg} states that in model (\ref{eq:model:regresshidden}), everything can be recovered based only on $(Y_{i})_{i\geq 1}$ in regression problems with dependent hidden regressors. In particular, Corollary~\ref{cor:reg} extends the identification results of \cite{dumont:lecorff:2017, dumont:lecorff:2017b} to the cases where the distribution of the additive noise is unknown. Numerical experiments in the case where $g: x\mapsto \cos x$ are given in Section~\ref{sec:simu}.


\section{Consistent Estimation}
\label{sec:estimation}

In this section, we propose two different estimation methods. The first one builds directly on the identifying equation (\ref{analyticmain}). It is similar to the method proposed in \cite{gassiat:rousseau:2016} for parametric estimation of the finite dimensional parameter (in their work, the hidden variables take finitely many values). The second estimation method is a likelihood method, which uses the Markov modeling of the hidden variables. The idea is to approximate the  continuous state space by a finite state space obtained by discretization, and to use  penalized likelihood to select automatically the number of points in the approximation and their location. 
Since the likelihood uses the joint distribution of all observations, likelihood estimation should be more accurate when the hidden process is indeed Markovian. Moreover, for hidden Markov models in which the distribution of the hidden variable has distribution in $\M_{\rho}$ for $2\leq \rho <3$, then one has to use an estimation method using the Markovian property. Thus in such a case, one can not use the method in Section \ref{sec:least:squares} but rather the likelihood method.  For both estimation methods, the identifiability theorem is the cornerstone to establish the consistency of the estimator.

\subsection{Using Least Squares for Characteristic Functions}
\label{sec:least:squares}
 In the following, objects related to the true (unknown) distribution $\po^{\star}$ of the observed process are denoted with the superscript $\star$.  Let $\Scal$ be a compact neighborhood of $0$ in $\R^{2d}$, and let $w
: \R^{d} \times \R^{d} \to \R_{+}$ 
be a positive function on $\Scal$. Let $\phi^\star$ be the characteristic function of $\varepsilon_1$.
For any probability distribution $R$ on $\R^{d} \times \R^{d}$, define
\begin{multline*}
M(R)=\int_\Scal
 \left\vert \Phi_{R^{\star}}(it_{1},it_{2})\Phi_{R}(it_{1},0)\Phi_{R}( 0, it_{2})  -
  \Phi_{R}(it_{1},it_{2})\Phi_{R^{\star}}(it_{1},0)\Phi_{R^{\star}}( 0, it_{2})
  \right\vert^{2} \\
  \vert   \phi^{\star}(t_{1})\phi^{\star}(t_{2})   \vert^{2} w(t_{1},t_{2})\rmd t_{1} \rmd t_{2}\eqsp.
\end{multline*}
Under appropriate assumptions, by the proof of Theorem \ref{theoident1}, $M(R)=0$ if and only if $R=R^{\star}=R_{K^{\star}}$ up to translation. Using an estimator $\widehat {\Phi}_{n}$ of the characteristic function of $(Y_{1},Y_{2})$, define an estimator of $M(\cdot)$ by
\begin{equation*}
M_{n}(R) \! = \!\!\int_\Scal
 \! \left\vert   \widehat {\Phi}_{n}(t_{1},t_{2})\Phi_{R}(it_{1},0)\Phi_{R}( 0, it_{2})  -
  \Phi_{R}(it_{1},it_{2})\widehat {\Phi}_{n}(t_{1},0)\widehat {\Phi}_{n}( 0, t_{2})
  \right\vert^{2}  \!\!\!\! w(t_{1},t_{2})\rmd t_{1} \rmd t_{2}\eqsp.
\end{equation*}
Let $\Rcal$ be a set of probability distributions on $\R^{d} \times \R^{d}$  such that for some $\rho <2$,  for all $R\in \Rcal$, both marginal distributions of
$R$ are in $\M_{\rho}$ and $R$ satisfies assumption H\ref{assum::}.
Define
$\widehat{R}_{n}$ as an element of $\Rcal$  satisfying
\begin{equation*}
M_{n}(\widehat{R}_{n}) = \inf_{R \in \Rcal} M_{n}(R).
\end{equation*}
Under the assumptions of Theorem \ref{theoconsi1}, $\widehat{R}_{n}$ exists but may be not uniquely defined because of translation invariance. 
Let $d$ be a distance that metrizes weak convergence on $\Rcal$, and define $Z_{n}(t_{1},t_{2})$ by
$
Z_{n}(t_{1},t_{2})=\sqrt{n} ( \widehat {\Phi}_{n}(t_{1},t_{2})-  \Phi_{R^{\star}}(it_{1},it_{2}) \phi_{1}^{\star}(t_{1})\phi_{2}^{\star}(t_{2}))
$.
\begin{theo}
\label{theoconsi1}
Assume that 
$
\sup_{(t_{1},t_{2})\in \Scal} | Z_{n}(t_{1},t_{2}) | = O_{\po^{\star}}(1)\eqsp.
$
Then,
\begin{equation*}
M(\widehat{R}_{n}) =  O_{\po^{\star}}(n^{-1/2}).
\end{equation*}
If moreover $\Rcal$ is compact for the weak convergence topology and $R^{\star}\in \Rcal$, then
$d(\widehat{R}_{n},\Rcal^{\star})$ tends to $0$ in $\po^{\star}$-probability as $n$ tends to infinity, where $\Rcal^{\star}$ is the set of $R\in \Rcal$ that are equal to $R^{\star}$ up to translation.
\end{theo}
In the parametric setting, the $n^{-1/2}$ rate on $M(\widehat{R}_{n})$ leads to a $n^{-1/2}$ rate on the parameter as proved  in \cite{gassiat:rousseau:2016}, where the main point  is to prove that the Hessian matrix of the criterion  is non singular at the true (unknown) parameter. However, obtaining nonparametric rates on $\widehat{R}_{n}$ from rates on $M(\widehat{R}_{n})$  is much more difficult, since in this infinite dimensional setting,  the Hessian operator can not have continuous inverse.

Note that consistency with respect to some topology is a consequence of the continuity of $M$ and the compactness of  $\cal R$ in the same topology. Consistency in other topologies could be derived under other assumptions. The proof of Theorem~\ref{theoconsi1} is postponed to  Appendix~\ref{sec:proof:ident1} for completeness.

\bigskip \noindent

\noindent {\em Comments on the assumptions of Theorem~\ref{theoconsi1}.}
\begin{enumerate}[i)]
\item \emph{Assumption on $Z_{n}$.} If $\widehat {\Phi}_{n}$ is the empirical estimator, then this assumption holds as soon as the hidden Markov chain is strongly mixing, see for instance \cite{doukhan:massart:rio:1994} and \cite{doukhan:massart:rio:1995}.
\item \emph{The marginals of each $R \in \Rcal$ are in $\Mcal_\rho$.}
For any positive $\rho$, $A$, $B$ and any positive integer $d'$, let $\M_{\rho,d'}(A,B)$ be the set of finite measures $\mu$ on $\R^{d'}$ such that for all $u\in \R^{d'}$,
$\int\exp \left(u^{T}x\right) \rmd \mu (x)
\leq A \exp \left( B \|u\|^{\rho}\right)$. 
For any $\rho>0$, $A > 0$ and $B > 0$, $\Mcal_{\rho,2d}(A,B)$ is compact for the weak convergence topology and for each distribution $R$ on $\R^d \times \R^d$ such that $R \in \Mcal_{\rho,2d}(A,B)$, both marginal distributions of $R$ are in $\Mcal_\rho$. Thus, it is enough to choose $\Rcal \subset \Mcal_{\rho,2d}(A,B)$ for some $\rho$, $A$ and $B$. In this case, the closure of $\Rcal$ is compact and still a subset of $\Mcal_{\rho,2d}(A,B)$.
\item \emph{All elements of $\Rcal$ satisfy assumption H\ref{assum::}.} 
A way to ensure this is to assume that there exists $\majZ > 0$ and $\underline{\varphi} : \C^d \rightarrow \R_+^*$ such that
\begin{equation}
\label{eq_strong_H1}
\forall R \in \Rcal, \quad
\forall z_0 \in \C^d, \quad
	\sup_{z_1 \in \C^d, \|z_1\| \leq \majZ} |\Phi_R(z_0,z_1)| \vee |\Phi_R(z_1,z_0)| \geq \underline{\varphi}(z_0) \eqsp.
\end{equation}
Note that since if $(R_n)_{n\geqslant 1}$ is a sequence of distributions in $\Mcal_{\rho,2d}(A,B)$ that converges to $R$ in distribution, then $\Phi_{R_n}$ converges to $\Phi_R$ uniformly over all compacts of $(\C^d)^2$ (because the set of functions $\{ \Phi_R : R \in \Mcal_{\rho,2d}(A,B)\}$ is pointwise equicontinuous),
the closure of $\Rcal \cap \Mcal_{\rho,2d}(A,B)$ also satisfies equation~\eqref{eq_strong_H1}.

Hence, the largest subset $\Rcal$ of $\Mcal_{\rho,2d}(A,B)$ (for some $\rho \in (0,2)$, $A>0$ and $B>0$) that contains only probability distributions satisfying equation~\eqref{eq_strong_H1} is compact for the weak convergence topology, each of its elements satisfies H\ref{assum::}, and provided that the parameters $\rho$, $A$, $B$, $\majZ$ and $\underline{\varphi}$ are suitable, it contains $R^\star$, thus it satisfies the assumptions of Theorem~\ref{theoconsi1}.
\end{enumerate}

\subsection{Using Maximum Likelihood}
\label{sec:mle}

In this section we fully exploit the Markovian structure of the latent variables.
Using the fact that continuous distributions may be approximated by discrete distributions, we consider finite state space hidden Markov models and the associated maximum likelihood estimator. The idea is to replace the (continuous) support of the hidden process  by a finite support. Increasing the number of  support points reduces the approximation error (the bias) while increasing the estimation error. Thus, a careful bias-variance trade-off has to be performed to obtain consistent estimators. We propose a penalized likelihood estimator that automatically selects the number of support points. Its consistency is obtained thanks to the identifiability Theorem \ref{theoident1} and to the oracle inequality proved in \cite{lehericy2018misspe}, Theorem~6. 

We assume in this section that the hidden process $(X_i)_{i \geq 1}$ takes values in a known compact set $\Lambda = [-L,L]^d \subset \R^d$ and that the distribution of the noise  is absolutely continuous with respect to the Lebesgue measure on $\R^d$. Denote by $K^{\star}$ the transition kernel of the hidden process, and by $\gamma^{\star}$ the density of the noise with respect to the Lebesgue measure. Since the compact is known, all possible kernels $K$ are such that $\mu_K \in \Mcal_1$. 

Transition kernels on finite sets are described by the number of points $r$ of their support, the vector $\Xfrak = (x_1, \dots, x_r)$ of their support points and the transition matrix $Q$ between these points: for all $(z,z') \in \{1, \dots, r\}^2, \  Q(z,z') = \po(X_1 = x_{z'} | X_0 = x_z)$.
For a vector $\Xfrak \in \Lambda^r$, a transition matrix $Q$ with stationary distribution $\mu_{Q}$ and a density $\gamma$, the log-likelihood of the parameter $(\Xfrak,Q,\gamma)$ given the observations $(Y_i)_{1 \leq i \leq n}$ is
\begin{equation*}
\ell_{n}(\Xfrak,Q,\gamma)=\log \left(\sum_{z_{1},\ldots,z_{n}\in \{1,\ldots,r\}}\!\!\!\!\!\!\!\!\!\!\mu_{Q}(z_{1})\gamma(Y_{1}-x_{z_{1}})\prod_{k=2}^{n}Q(z_{k-1},z_{k})\gamma(Y_{k}-x_{z_{k}})
\right)\eqsp.
\end{equation*}
In this section, a penalized likelihood function is used to perform model selection. Consider a collection of finite dimensional models $(S_{r,D,n})_{r \geq 1, D \geq 1, n \geq 1}$, in which $D$ is a complexity parameter related to the sieve approximation of the nonparametric set in which $\gamma$ lies. Then, for each $r \geq 1$ and $D \geq 1$, the maximum likelihood estimator of model $S_{r,D,n}$ is defined by
\begin{equation*}
(\hat{\Xfrak}_{r,D,n},\hat{Q}_{r,D,n},\hat{\gamma}_{r,D,n})
	\in \argmax_{(\Xfrak,Q,\gamma) \in S_{r,D,n}}
			\frac{1}{n} \ell_n(\Xfrak,Q,\gamma)\eqsp.
\end{equation*}
The number of states and the model dimension are selected using the penalized likelihood:
\begin{equation*}
(\hat{r}_n, \hat{D}_n) \in \argmax_{r \leq \log n, D \leq n}
	\left( \frac{1}{n} \ell_n(\hat{\Xfrak}_{r,D,n},\hat{Q}_{r,D,n},\hat{\gamma}_{r,D,n}) - (D+r^2) \frac{(\log n)^{15}}{n} \right)
\end{equation*}
and the final estimators are defined as
\begin{equation*}
(\hat{\Xfrak}_n, \hat{Q}_n, \hat{\gamma}_n) = \left(\hat{\Xfrak}_{\hat{r}_n, \hat{D}_n, n}, \hat{Q}_{\hat{r}_n, \hat{D}_n, n}, \hat{\gamma}_{\hat{r}_n, \hat{D}_n, n}\right)\eqsp.
\end{equation*}
The specific form of the penalty is chosen according to the theory developed in  \cite{lehericy2018misspe}, but in practice the slope heuristics as in \cite{MR2865029} could be used to calibrate the penalty.

The nonparametric set $\Gamma$ of possible noise densities is described now as a set of nonparametric mixtures. Then, the finite dimensional sieve is given by finite mixtures with at most $D$ support points. 
Let $\Theta$ be a compact subset of $\R^d \times GL_d(\R)$ and $f : y \in \R^d \longmapsto (2\pi)^{-d/2} \exp(- \|y\|^2 / 2)$ be the density of a standard multivariate normal distribution. Write $\Pcal(\Theta)$ the set of probability measures on $\Theta$, let
\begin{equation}
\label{eq_ex_Gamma}
\Gamma = \left\{ \gamma : y \longmapsto \int_\Theta |\det(\Sigma)| f\left( \Sigma(y-\mu) \right) \rmd p(\mu,\Sigma) : p \in \Pcal(\Theta), \int_\Theta \mu \rmd p(\mu, \Sigma) = 0 \right\}
\end{equation}
be the set of densities of location-scale mixtures of $f$ with parameters in $\Theta$. The condition $\int_\Theta \mu \rmd p(\mu, \Sigma) = 0$ ensures that all densities in $\Gamma$ are centered. For $(\mu,\Sigma) \in \Theta$, write $\delta_{\mu,\Sigma}$ the Dirac measure centered on $(\mu,\Sigma)$.
Let $(G_D)_{D \geq 1}$ be  defined for all $D \geq 1$ by
\begin{equation*}
G_D = \left\{ \gamma : y \longmapsto \sum_{i=1}^D p_i \det(\Sigma_i) f\left( \Sigma_i(y-\mu_i) \right) : \sum_{i=1}^D p_i \delta_{(\mu_i, \Sigma_i)} \in \Pcal(\Theta), \sum_{i=1}^D p_i \mu_i = 0 \right\}.
\end{equation*}
%
Transition kernels are understood as functions from $\Lambda$ to $\Pcal(\Lambda)$ endowed with the weak convergence topology. 
For $p\geq 1$,  let $W_p(\mu_1,\mu_2)$ be the Wasserstein distance  between two probability measures  $\mu_{1}$ and $\mu_{2}$ on the same Euclidian space $E$ which is defined as the infimum of  $(\int_{E\times E} \|x-y\|^{p} \pi (\rmd x,\rmd y))^{1/p} $ over the set of probabilities $\pi$ such that $\mu_{1}=\int \pi(\cdot, \rmd y)$ and $\mu_{2}=\int \pi(\rmd x,\cdot)$,
see \cite{MR1619171} or \cite{MR2459454}. Wasserstein distances are convenient to compare probability measures that may be singular to each other and  $W_1$ metrizes the weak convergence topology for compactly supported distributions.
 It is assumed that all kernels used in the proposed procedure share the same modulus of continuity  $\omega$. It is possible to assume that $\omega$ is a concave function with no loss of generality since $\Pcal(\Lambda)$ has finite $W_1$-diameter. Let $C \geq 2$ be a constant.
\begin{hypH}
\label{hyp_transitionKernel}
The application $x \in \Lambda \longmapsto K^\star(x, \cdot) \in (\Pcal(\Lambda), W_1)$ admits the modulus of continuity $\omega/2$ and there exists a probability measure $\lambda^\star$ on $\Lambda$ such that for all $x \in \Lambda$, $K^\star(x, \cdot)$ has a density with values in $[2/C,C/2]$ with respect to $\lambda^\star$.
\end{hypH}
The collection of models $(S_{r,D,n})_{r \geq 1, D \geq 1, n\geq 1}$ used in the maximum likelihood estimation is defined as follows. 
For all $r \geq 1$ and $D \geq 1$, let $\Sbf_{r,D}$ be the set of all $(\Xfrak,Q,\gamma) \in \Lambda^r \times [1/(Cr), C/r]^{r \times r} \times G_D$ such that $Q$ is a transition matrix and the transition kernel $x_z \longmapsto \sum_{z'=1}^r Q(z,z') \delta_{x_{z'}}$ admits the modulus of continuity $\omega$ with respect to $W_1$. 

In order to state the consistency result, a continuous kernel associated with the discrete kernels of the models has to be introduced. For $(\Xfrak,Q,\gamma) \in \Sbf_{r,D}$, denote by $K_{\Xfrak,Q}$ a transition kernel on $\Lambda$ that admits the modulus of continuity $\omega$ with respect to the Wasserstein 1 metric, extends the kernel defined by $Q$ on $\{x_z\}_{z = 1, \dots, r}$ and such that the support of $K_{\Xfrak,Q}(x,\cdot)$ is in $\{x_z\}_{z = 1, \dots, r}$ for all $x \in \Lambda$. Linear interpolation provides a way to construct such a kernel as soon as the modulus $\omega$ is concave.

To conclude the definition of the models, let $\majZ' > 0$, $(\majZ_n)_n$ be a sequence of positive real numbers such that $\majZ_n \rightarrow +\infty$ and $\underline{\varphi} : \C^d \rightarrow \R_+^*$. For all $r,D,n$, let $S_{r,D,n}$ be the subset of $\Sbf_{r,D}$ such that
\begin{multline}
\label{eq_strong_H1_seq}
\forall (\Xfrak,Q,\gamma) \in S_{r,D,n}\eqsp, \quad
\forall z_0 \in \C^d \text{ s.t. } \|z_0\| \leq \majZ_n\eqsp, \\
	\sup_{z_1 \in \C^d, \|z_1\| \leq \majZ'} |\Phi_{R_{K_{\Xfrak,Q}}}(z_0,z_1)| \vee |\Phi_{R_{K_{\Xfrak,Q}}}(z_1,z_0)| \geq \underline{\varphi}(z_0) \eqsp.
\end{multline}
This is a relaxed version of H\ref{assum::} and equation~\eqref{eq_strong_H1} in the sense that eventhough the elements of $S_{r,D,n}$ may not satisfy H\ref{assum::}, the limit of a convergent sequence $(R_{K_{\Xfrak_n,Q_n}})_n$ with $(\Xfrak_n,Q_n,\gamma_n) \in S_{r,D,n}$ for all $n$ satisfies H\ref{assum::}. For the following theorem to work, $\majZ'$ can be chosen arbitrarily, $(\majZ_n)_n$ must grow ``slowly enough" and $\underline{\varphi}$ be ``small enough"; an appropriate choice of these quantities is discussed in the proof of the Theorem, see Appendix~\ref{sec:proof:generaltheorem}.

\begin{theo}
\label{theoconsi2}
Assume that assumptions H\ref{assum::} and 
H\ref{hyp_transitionKernel} hold for $K^\star$. Assume also that $\gamma^* \in \Gamma$.
Let $\lambda^\star$ be the measure defined in assumption H\ref{hyp_transitionKernel} and $\Supp(\lambda^\star)$ its support. Then,  
almost surely, the maximum likelihood estimator satisfies
\begin{equation*}
\sup_{x \in \Supp(\lambda^\star)} W_1(K_{\hat{\Xfrak}_n,\hat{Q}_n}(x,\cdot),K^{\star}(x, \cdot)) \underset{n \rightarrow \infty}{\longrightarrow} 0
\end{equation*}
and
$\| \hat{\gamma}_n - \gamma^\star \|_1 \underset{n \rightarrow \infty}{\longrightarrow} 0\eqsp.
$
In particular, almost surely under $\po^\star$, for all $x \in \Supp(\lambda^\star)$, $K_{\hat{\Xfrak}_n,\hat{Q}_n}(x,\cdot) 
\longrightarrow 
K^{\star}(x, \cdot)$ for the weak convergence topology and 
$\mu_{K_{\hat{\Xfrak}_n,\hat{Q}_n}} 
\longrightarrow
\mu_{K^{\star}}$ for the weak convergence topology.
\end{theo}
Theorem \ref{theoconsi2} is a special case of a theorem stated and proved in Appendix~\ref{sec:proof:mle} that holds for more general sets $\Gamma$ and $(G_D)_{D \geq 1}$.

\section{Simulations}
\label{sec:simu}

Consider the model where $Z_0$ is a uniform random variable on $(0,2\pi)$ and for all $k\geq 1$,
\begin{equation*}
Z_{k} = Z_{k-1} + \sigma_x \eta_k \eqsp,\quad X_{k} = \cos\left(Z_{k}\right)\quad\mathrm{and}\quad
Y_k = X_k + \sigma_y \varepsilon_k\eqsp,
\end{equation*}
where $(\sigma_x,\sigma_y) \in \times \R_+^*\times \R_+^*$ and where $(\varepsilon_k,\eta_k)_{k\geq 1}$ are independent standard Gaussian random variables independent of $Z_0$. The parameters $(\sigma_x,\sigma_y) = (0.1,0.1)$ are used to sample the observations. Assumption H\ref{hyp_transitionKernel} holds: the transition kernel $K^\star$ of $(X_k)_{k \geq 1}$ is $1/2$-H\"older and the probability measure $\lambda^\star$ can be taken as the invariant measure of $K^\star$.

This section provides numerical illustrations of the maximum likelihood approach, additional simulations using least squares for the characteristic functions are given in Appendix~\ref{sec:simu:appendix}.  The algorithm proposed here  is more efficient than the algorithm proposed in Appendix~\ref{sec:simu:appendix} whose performance  highly depends on the evolutionnary algorithm to minimize the criterion. The performance of the estimation procedure proposed in Section~\ref{sec:mle} is assessed in the case where $\Lambda = \R$ and $\Gamma$ is as in \eqref{eq_ex_Gamma} with $\Theta = \R \times (0,+\infty)$.
Although the compactness assumptions of Section~\ref{sec:mle} are not satisfied, in practice, the estimator is shown to converge to the true distribution. 
The main reason for these assumptions is to ensure theoretical consistency by ruling out the worst case scenarios where the estimators are degenerate.

For each $n \in \{5.10^3, 10^4, 2.10^4, 5.10^4, 10^5, 2.10^5\}$, 10 independent and identically distributed sequences $(Y_i)_{i = 1, \dots, n}$ are generated. For each sample, an approximation of the maximum likelihood estimator is computed using the Estimation Maximization algorithm \cite{dempster:laird:rubin:1977} for $D = 2$ and $r \in \{10, 20, 30\}$.
The error criterion is the estimated Wasserstein distance between the estimated and the true distribution of $(X_1, X_2)$, computed using $N_X \times N_W$ independent and identically distributed pairs $(X_{1,i}^{(j)}, X_{2,i}^{(j)})_{i = 1, \dots, N_X, j = 1, \dots, N_W}$ following the distribution $R_{K^\star}$ with $N_X = 5000$ and $N_W = 4$:
\begin{equation}
\label{eq:wass:dist}
\text{Error}(\hat{\Xfrak}_n, \hat{Q}_n) =
	\frac{1}{N_W} \sum_{j=1}^{N_W}
	W_1\left(  
		R_{K_{\hat{\Xfrak}_n, \hat{Q}_n}}, 
		\frac{1}{N_X} \sum_{i=1}^{N_X} \delta_{(X_{1,i}^{(j)}, X_{2,i}^{(j)})}
	\right),
\end{equation}
or equivalently (when written as a distance between weighted point processes)
\begin{equation*}
\text{Error}(\hat{\Xfrak}_n, \hat{Q}_n) =
	\frac{1}{N_W} \sum_{j=1}^{N_W}
	W_1\left(  
		\sum_{x,x' \in \hat{\Xfrak}_n} R_{K_{\hat{\Xfrak}_n, \hat{Q}_n}}(x,x') \delta_{(x,x')}, 
		\frac{1}{N_X} \sum_{i=1}^{N_X} \delta_{(X_{1,i}^{(j)}, X_{2,i}^{(j)})}
	\right).
\end{equation*}
The distance $W_1$ is computed using function \verb?wasserstein? from R package \verb?transport? \cite{Rpackage_transport, R}. The results are displayed in Figure~\ref{fig:mle}.

\newpage

\begin{figure}[!ht]
\centering
\includegraphics[width=.8\textwidth]{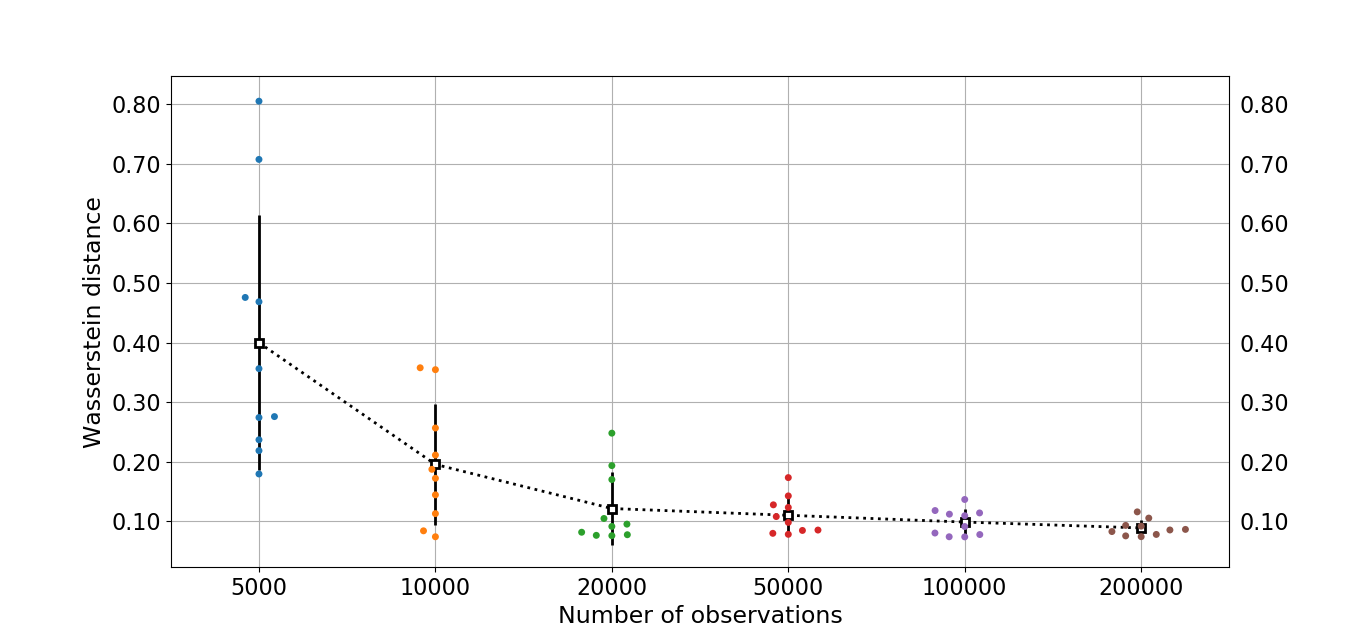}
\includegraphics[width=.8\textwidth]{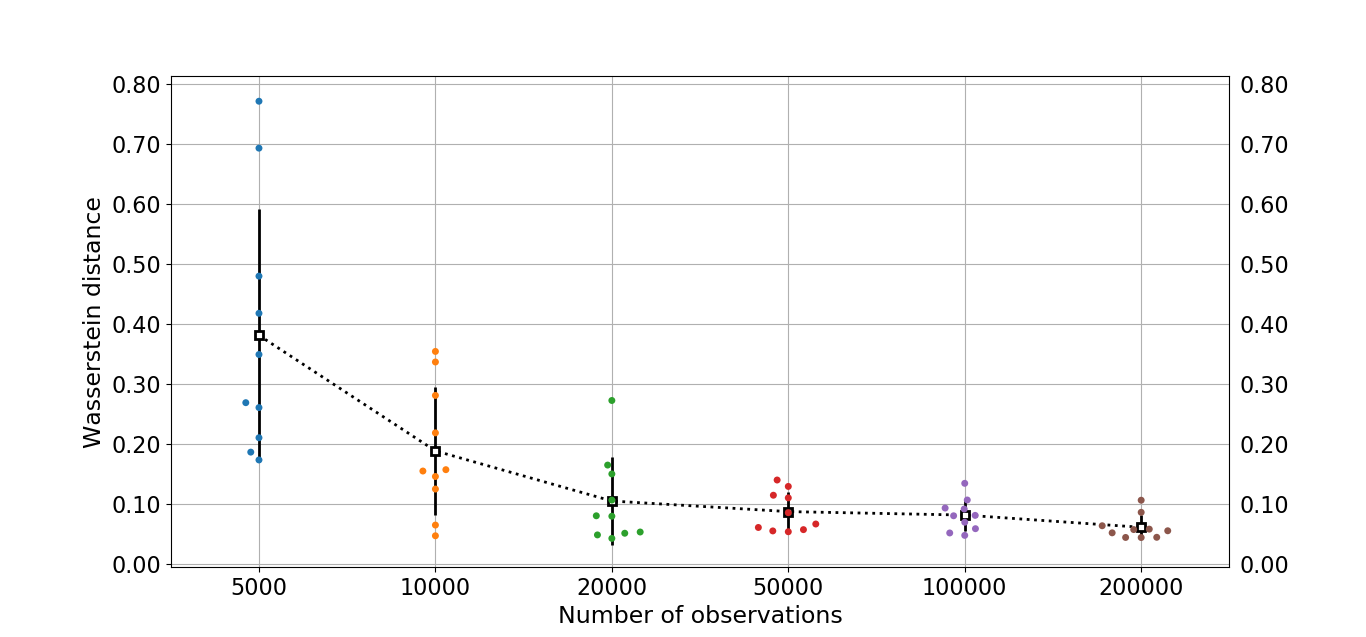}
\includegraphics[width=.8\textwidth]{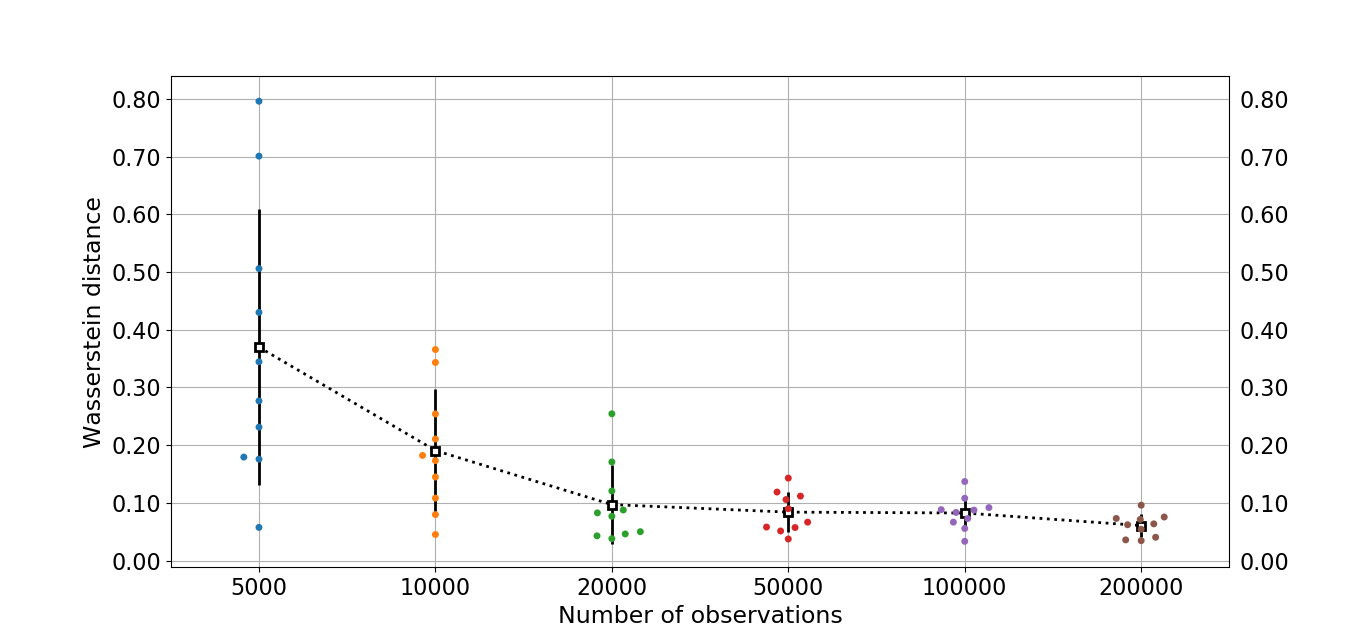}
\caption{Wasserstein distance computed as in \eqref{eq:wass:dist} for $r=10$ (top), $r = 20$ (middle) and $r = 30$ (bottom).  Each dot is an estimated value with the maximum likelihood approach. For each value of $r$, the mean value (squares) over all runs as well as the empirical standard deviation (bars) are displayed.}
\label{fig:mle}
\end{figure}

\begin{figure}[!ht]
\centering
\begin{tabular}{c|ccc}
$n$ & $r=10$ & $r=20$ & $r=30$ \\ \hline
5000 & [0.196; 0.327; 0.773] & [0.170; 0.312; 0.778] & [0.059; 0.304; 0.800] \\ 
10000 & [0.075; 0.182; 0.355] & [0.047; 0.169; 0.363] & [0.045; 0.184; 0.371] \\ 
20000 & [0.075; 0.097; 0.261] & [0.045; 0.082; 0.267] & [0.036; 0.079; 0.255] \\ 
50000 & [0.077; 0.098; 0.166] & [0.048; 0.077; 0.155] & [0.034; 0.074; 0.160] \\ 
100000 & [0.076; 0.103; 0.149] & [0.046; 0.091; 0.142] & [0.038; 0.084; 0.139] \\ 
200000 & [0.076; 0.087; 0.110] & [0.045; 0.065; 0.100] & [0.037; 0.062; 0.107] \\
\end{tabular}
\caption{Summary of the Wasserstein distance computed as in \eqref{eq:wass:dist}. Each cell contains the min, median and max value of the error over the 10 simulations with corresponding $r$ and $n$.}
\end{figure}

\section{Discussion}
\label{sec:discussion}
In this paper, we proved that statistical learning of a Markov signal corrupted by additive noise is possible without any knowledge of the noise and with weak nonparametric assumptions on the distribution of the hidden variables. 
We proposed estimation methods and proved consistency of the estimators under weak assumptions.
Establishing rates of convergence is a much more challenging task. It would require to relate the limiting criterion  to the risk of the estimator, namely to lower bound $M(\cdot)$ in section  \ref{sec:least:squares}
for the least squares estimator, or the entropy rate $\Kbf$ in 
section \ref{sec:proof:generaltheorem} for the maximum likelihood estimator, to some risk on the kernel of the hidden process and some risk on the distribution of the noise. In other words, the challenge is to get quantitative control in the inverse problem relating the kernel $K$ and the distribution $P$ to $\po_{K,P}$, in a similar way as Theorem 6 in \cite{MR3543517}. This can not be done as in usual deconvolution where one has prior knowledge on the noise distribution. 

As explained in Section \ref{sec:identifiability:general}, identifiability may be proved in other dependency settings. In the context of at least $2$-dimensional observations, deconvolution is possible without any knowledge of the noise distribution if the signal to be recovered has dependent coordinates. This is further developed in 
\cite{GLL}. In this work, rates are provided for the least squares estimator in the setting of independent and identically distributed observations.

\bibliographystyle{apalike}
\bibliography{biblio}

\appendix

\section{Proof of Theorem~\ref{theoident1}, Theorem~\ref{theoident2} and Theorem~\ref{theoconsi1}}
\label{sec:proof:ident1}

\subsection{Proof of Theorem~\ref{theoident1}}
The following result, which  may be established by arguing variable by variable, is used repeatedly in this proof. If a multivariate function is analytic on the whole multivariate complex space and is the null function in an open set of the multivariate real space or in an open set of the multivariate purely imaginary space, then it is the null function on the whole multivariate complex space. 

Assume that  $\po_{{K},P}=\po_{{\widetilde K},\widetilde P}$ and let $\phi$ (resp. $\widetilde \phi$) be the characteristic function of $\noisedist$ (resp. $\noisedistbis$). Notice that $\Phi_{R_K}(it_{1},it_{2})$ (resp. $\Phi_{R_{\tilde K}}(it_{1},it_{2})$) for real numbers numbers $t_{1}$ and $t_{2}$ defines the characteristic function of $(X_{1},X_{2})$ when the Markov chain has kernel $K$ (resp. $\tilde{K}$) and $\Phi_{R_K}(it,0)=\Phi_{R_K}(0,it)$ (resp. $\Phi_{R_{\tilde{K}}}(it,0)=\Phi_{R_K}(0,it)$)  for real numbers $t$ defines the characteristic function of any $X_{i}$ when the Markov chain has kernel $K$ (resp. $\tilde{K}$).
Since the distribution of $Y_{1}$ and $Y_2$  are the same under $\po_{{K},P}$ and $\po_{{\widetilde K},\widetilde P}$, for any $t\in \R^{d}$,
\begin{equation}
\label{chara1bis}
\phi\left(t\right)\Phi_{R_K}\left(it,0\right) = \widetilde{\phi}\left(t\right)\Phi_{R_{\tilde K}}\left(it,0\right)\eqsp.
\end{equation}
Since the distribution of $(Y_{1},Y_{2})$  is the same under $\po_{{K},P}$ and $\po_{{\widetilde K},\widetilde P}$, for any $(t_{1},t_{2})\in \R^{d}\times \R^{d}$,
\begin{equation}
\label{chara3bis}
\phi\left(t_{1}\right)\phi\left(t_{2}\right)\Phi_{R_K}\left(it_{1},it_{2}\right)=\widetilde{\phi}\left(t_{1}\right)\widetilde{\phi}\left(t_{2}\right)\Phi_{R_{\tilde K}}\left(it_{1},it_{2}\right)\eqsp.
\end{equation}
There exists a neighborhood $V$ of $0$ in $\R^{d}\times \R^{d}$ such that for all $t=(t_{1},t_{2})\in V$, $\phi\left(t_{1}\right)\neq 0$,  $\phi\left(t_{2}\right)\neq 0$, $\widetilde{\phi}\left(t_{1}\right)\neq 0$, $\widetilde{\phi}\left(t_{2}\right)\neq 0$, so that  (\ref{chara1bis}) and (\ref{chara3bis}) imply that for any $ (t_{1},t_{2})\in V^{2}$,
\begin{equation*}
\Phi_{R_K}\left(it_{1},it_{2}\right)\Phi_{R_{\tilde K}}\left(it_{1},0\right)\Phi_{R_{\tilde K}}\left(it_{2},0\right)=\Phi_{R_{\tilde K}}\left(it_{1},it_{2}\right)\Phi_{R_K}\left(it_{1},0\right)\Phi_{R_K}\left(it_{2},0\right)\eqsp.
\end{equation*}
Since $(z_1,z_2) \mapsto \Phi_{R_K}\left(z_1,z_2\right)\Phi_{R_{\tilde K}}\left(z_1,0\right)\Phi_{R_{\tilde K}}\left(z_2,0\right)-\Phi_{R_{\tilde K}}\left(z_1,z_2\right)\Phi_{R_K}\left(z_1,0\right)\Phi_{R_K}\left(z_2,0\right)$ is a multivariate analytic function of $2d$ variables which is zero in a purely imaginary neighborhood of $0$, then it is the null function on the whole multivariate complex space so that for any $z_{1}\in\C^{d}$ and $z_{2}\in\C^{d}$,
\begin{equation}
\label{analyticbis}
\Phi_{R_K}\left(z_{1},z_{2}\right)\Phi_{R_{\tilde K}}\left(z_{1},0\right)\Phi_{R_{\tilde K}}\left(z_{2},0\right)=\Phi_{R_{\tilde K}}\left(z_{1},z_{2}\right)\Phi_{R_K}\left(z_{1},0\right)\Phi_{R_K}\left(z_{2},0\right)
\eqsp.
\end{equation}
Fix $(u_{2},\ldots,u_{d})\in \C^{d-1}$ and let  $\Zsf$ be the set of zeros of $u \mapsto \Phi_{R_{K}}(u,u_{2},\ldots,u_{d},0)$ and $\widetilde{\Zsf}$ be the set of zeros of $u\mapsto\Phi_{R_{\tilde K}}(u,u_{2},\ldots,u_{d},0)$. Let  $u_{1}\in \Zsf$ and write $z_{1}=(u_{1},u_{2},\ldots,u_{d})$ so that by \eqref{analyticbis}, for any $z_{2}\in\C^{d}$,
\begin{equation}
\label{zerosbis}
\Phi_{R_K}\!\left(z_{1},z_{2}\right)\Phi_{R_{\tilde K}}\!\left(z_{1},0\right)\Phi_{R_{\tilde K}}\!\left(z_{2},0\right)=0\;\mathrm{and}\; \Phi_{R_K}\!\left(z_{2},z_{1}\right)\Phi_{R_{\tilde K}}\!\left(z_{1},0\right)\Phi_{R_{\tilde K}}\!\left(z_{2},0\right)=0
\eqsp.
\end{equation}
By assumption H\ref{assum::}, $z_2 \rightarrow \Phi_{R_K}\left(z_{1},z_{2}\right)$ is not the null function or $z_2 \rightarrow \Phi_{R_K}\left(z_{2},z_{1}\right)$ is not the null function.  Assume without loss of generality that $z_2 \rightarrow \Phi_{R_K}\left(z_{1},z_{2}\right)$ is not the null function (the proof follows the same steps in the other case). 
Then, there exists $z_{2}^{\star}$ in $\C^{d}$ such that $\Phi_{R_K}\left(z_{1},z_{2}^{\star}\right)\neq 0$ and by continuity, there exists an open neighborhood of $z_{2}^{\star}$ such that for all $z_{2}$ in this open set, $\Phi_{R_K}\left(z_{1},z_{2}\right)\neq 0$. Since $z\mapsto \Phi_{R_{\tilde K}}\left(z,0\right)$ is not the null function and is analytic on $\C^{d
}$, it can not be null all over this open set, so that there exists $z_{2}$ such that simultaneously $\Phi_{R_K}\left(z_{1},z_{2}\right)\neq 0$ and $\Phi_{R_{\tilde K}}\left(z_{2},0\right)\neq 0$. 
 Then (\ref{zerosbis}) leads to $\Phi_{R_{\tilde K}}\left(z_{1},0\right)=0$, so that $\Zsf \subset \widetilde{\Zsf}$.
A symmetric argument yields $\widetilde{\Zsf}\subset \Zsf$ so that  $\Zsf = \widetilde{\Zsf}$.

Moreover, the analytic functions $u\mapsto \Phi_{R_K}(u,u_{2},\ldots,u_{d},0)$ and $u\mapsto \Phi_{R_{\tilde K}}(u,u_{2},\ldots,u_{d},0)$ have exponential growth order less than $2$, so that using Hadamard's factorization Theorem, see \cite[Chapter~5, Theorem~5.1]{Stein:complex}, there exists a polynomial function $s$ with  degree at most $1$ (and with coefficients depending on $(u_{2},\ldots,u_{d})$) such that for all $u\in\C$,
\[
\Phi_{R_K}(u,u_{2},\ldots,u_{d},0) = \rme^{s(u)}\Phi_{R_{\tilde K}}(u,u_{2},\ldots,u_{d},0)\eqsp.
\]
Arguing similarly for all variables, we get that there exists a  function $S$ on $\C^{d}$ which is, for any $i=1,\ldots,d$, polynomial with degree at most $1$ in $u_i$, and such that for all $(u_{1},\ldots,u_{d})\in\C^{d}$,
\begin{equation}
\label{eq:phi:exp}
\Phi_{R_K}(u_{1},u_{2},\ldots,u_{d},0) = \rme^{S(u_{1},u_{2},\ldots,u_{d})}\Phi_{R_{\tilde K}}(u_{1},u_{2},\ldots,u_{d},0)\eqsp.
\end{equation}
In other words, there exists complex functions $a_{i}$, $b_{i}$ on $\C^{d-1}$ such that, if we denote $u^{(-i)}$ the $(d-1)$-dimensional complex vectors with the same coordinates as $u$ except that $u_i$ is not included in the coordinates, then
$$
S(u_{1},u_{2},\ldots,u_{d})=a_{i}(u^{(-i)})u_{i}+b_{i}(u^{(-i)}),\;i=1,\ldots,d.
$$
But, for $i\neq j$, the fact that $a_{i}(u^{(-i)})u_{i}+b_{i}(u^{(-i)})=a_{j}(u^{(-j)})u_{j}+b_{i}(u^{(-j)})$ implies that $a_{i}(u^{(-i)})$ and $b_{i}(u^{(-i)})$ are polynomial functions with degree at most $1$ in $u_j$ (this may be seen for instance by taking complex derivatives), and by induction we get that $S$ is a polynomial function which is, for any $i=1,\ldots,d$ polynomial with degree at most $1$ in $u_i$.

Since $\Phi_{R_K}(0,\ldots,0)=\Phi_{R_{\tilde K}}(0,\ldots,0)=1$, the constant term of the polynomial $S$ is $0$. 
Assume that $\mu_{\tilde{K}}$  is not supported by $0$. Then there exist $a=(a_{1},\ldots,a_{d})\in\R^{d}$, $\alpha>0$ and $\delta >0$ such that 
$$
0 \notin \prod_{j=1}^{d} [a_{j}-\alpha,a_{j}+\alpha] \quad\mathrm{and}\quad
\mu_{\tilde{K}}\left(\prod_{j=1}^{d} [a_{j}-\alpha,a_{j}+\alpha]
\right)\geq \delta\eqsp,
$$
which gives, for all $u\in\R^{d}$,
$$ 
\Phi_{R_{\tilde K}}(u,0) \geq \delta \rme^{ \sum_{j=1}^{d}\inf_{x_{j} \in [a_{j}-\alpha,a_{j}+\alpha]}u_{j}x_{j}}\eqsp,
$$
so that using (\ref{eq:phi:exp}), for all $u\in\R^{d}$,
$$
\Phi_{R_{K}}(u,0) \geq \delta \rme^{S(u)} \rme^{ \sum_{j=1}^{d}\inf_{x_{j} \in [a_{j}-\alpha,a_{j}+\alpha]}u_{j}x_{j}}\eqsp.
$$
 If  $S$ has degree at least $2$, then there exist $i\neq j$ and polynomial functions with degree at most one in each variable $c_{1}$ on $\C^{d-2}$ and $c_{2}$, $c_{3}$ on $\C^{d-1}$ such that, if we denote $u^{(-i,-j)}$ the $(d-2)$-dimensional complex vectors with the same coordinates as $u$ except that $u_i$ and $u_{j}$ are not included in the coordinates, then
$S(u)=c_{1}(u^{(-i;-j)})u_{i}u_{j} + c_{2}(u^{(-i)}) + c_{3}(u^{(-j)}) $. Without loss of generality say that $i=1$ and $j=2$. Then it is possible to find $u\in \R^{d}$ and $\tilde{\delta}>0$ such that  for all $t\geq 0$,
$S(tu_{1},tu_{2},u_{3},\ldots,u_{d})\geq \tilde{\delta} t(u_{1}^2+u_{2}^{2})$ leading to
$$
\forall t\geq 0,\; \Phi_{R_{K}}(tu_{1},tu_{2},u_{3},\ldots,u_{d},0) \geq \delta\rme^{\tilde{\delta} t(u_{1}^2+u_{2}^{2})} \rme^{ \sum_{j=1}^{d}\inf_{x_{j} \in [a_{j}-\alpha,a_{j}+\alpha]}u_{j}x_{j}}\eqsp,
$$
contradicting the assumption that $\mu_{ K} \in \M_{\rho}$ for some $\rho <2$. Thus, 
 $S$ has degree at most $1$
 and there exists $m\in \C^{d}$ such that for all $z\in\C^{d}$, 
\begin{equation}
\label{idem}
\Phi_{R_K}(z,0)=\rme^{m^{T}z}\Phi_{R_{\tilde K}}(z,0)\eqsp.
\end{equation}
As for all $z\in\R^{d}$, $\Phi_{R_K}(-iz,0)=\overline{\Phi_{R_K}(iz,0)}$ and
$\Phi_{R_{\tilde K}}(-iz,0)=\overline{\Phi_{R_{\tilde K}}(iz,0)}$,  then $m\in\R^{d}$. Combining (\ref{idem}) with  (\ref{analyticbis}) yields, for all $ (t_{1},t_{2})\in \R^{d}\times\R^{d}$,
\begin{equation}
\label{ident1m}
\Phi_{R_K}(it_{1},it_{2})=e^{im^{T}t_{1}+im^{T}t_{2}}\Phi_{R_{\tilde K}}(it_{1},it_{2})\eqsp.
\end{equation} 
Then, using (\ref{chara1bis}), for all $t\in \R^{d}$ such that $\Phi_{R_K}(it,0)\neq 0$,  $\phi(t)=\rme^{-im^{T}t}\widetilde{\phi}(t)$. Since the set of zeros of $t\mapsto \Phi_{R_K}(it,0)$ has empty interior, for each $t$ such that $\Phi_{R_K}(it,0)= 0$ it is possible to find a sequence $(t_{n})_{n\geq 1}$ such that $t_{n}$ tends to $t$ and for all $n$,  $\Phi_{R_K}(it_{n},0)\neq 0$. But $\phi$ and $\widetilde{\phi}$ are continuous functions, so that for all $t\in\R$, 
\begin{equation}
\label{ident2m}
\phi(t)=\rme^{-im^{T}t}\tilde{\phi}(t)\eqsp.
\end{equation} 
The proof is concluded by noting that (\ref{ident1m}) and  (\ref{ident2m}) imply that $R_K=R_{\tilde K}$ and $P=\tilde{P}$ up  to translation.

\subsection{Proof of Theorem~\ref{theoident2}}
\label{sec:proof:ident2}

Following the same steps as in the proof of Theorem \ref{theoident1}, there exists a polynomial $S$ with real coefficients and degree at most $2$ such that, for all $z\in\C$, 
$
\Phi_{R_K}(z,0)=\rme^{S(z)}\Phi_{R_{\widetilde K}}(z,0)
\eqsp,
$
and for all $(z_{1},z_{2})\in \C\times \C$,
\begin{equation}
\label{eq:rho2}
\Phi_{R_K}(z_{1},z_{2})=e^{S(z_{1})}e^{S(z_{2})}\Phi_{R_{\widetilde K}}(z_{1},z_{2})\eqsp.
\end{equation}
Assume that $S$ has degree equal to $2$. Then, there exist real numbers $a$, $b$, $c$ such that for all $z\in\C$, $S(z)=az^{2}+bz+c$. With no loss of generality assume that $a>0$ (otherwise, replace $K$ by $\widetilde K$).
Then, (\ref{eq:rho2}) means that there exist independent and identically distributed Gaussian variables $\eta_{i}$, with variance $2a$, such that, if $(X_{i})_{i\geq1}$ is a stationary Markov chain with transition kernel $K$ and $(\tilde{X}_{i})_{i\geq1}$ is a stationary Markov chain with transition kernel $\widetilde K$,   $(X_{i})_{i\geq 1}$ has the same distribution as $(\tilde{X}_{i}+\eta_{i})_{i\geq 1}$, with $\eta_{i}$, $i\geq 1$, independent of $(\tilde{X}_{i})_{i\geq1}$. Using Lemma \ref{lem:gaussMarkov}, this implies that the random variables $(X_{i})_{i\geq 1}$ are independent and identically distributed. 
If $z\mapsto\Phi_{R_K}(z,0)$ has no zeros, then it has the same set of zeros as the constant function equal to one (corresponding to deterministic independent variables equal to $0$), so that 
using Hadamard's Theorem, since $\mu_{K}\in \M_{\rho}$  with $\rho<3$, then there exists a polynomial with degree at most $2$ such that $\Phi_{R_K}(z,0)$ is the exponential of that polynomial, so that $(X_{i})_{i\geq 1}$ is a sequence of independent Gaussian variables,
contradicting the assumption of Theorem \ref{theoident2}.
But for all $(z_{1},z_{2})\in\C^{d}\times \C^{d}$, $\Phi_{R_K}(z_{1},z_{2})=\Phi_{R_K}(z_{1},0)\Phi_{R_K}(0,z_{2})$.  Thus if $z\mapsto\Phi_{R_K}(z,0)$ has at least one zero  $z_{0}$, then 
$\Phi_{R_K}(z_{0},z)=0$ and $\Phi_{R_K}(z,z_{0})=0$ for all $z\in\C$, contradicting assumption H\ref{assum::} in Theorem \ref{theoident2}.
Then we may conclude that $S$ has degree at most $1$, and the end of the proof of Theorem \ref{theoident2} follows the same steps as the proof of Theorem \ref{theoident1}.

\subsection{Proof of Lemma \ref{lem:gaussMarkov}}

For all $x\in\R$, let $x'\mapsto q(x,x')$ be the density of the transition kernel 
of the Markov chain $(X_{i})_{i\geq 1}$ with respect to the Lebesgue measure  and $\mu$ be its stationary density. Denote $m$ the mean and $\sigma^{2}$ the variance of $\eta_{1}$, and let $\phi$ be the density of $\eta_{1}$ . Denote by $g_{1}$ the density of $X_{i}+\eta_{i}$, $g_{2}$ the density of $(X_{i}+\eta_{i}, X_{i+1}+\eta_{i+1})$, $g_{3}$ the density of $(X_{i}+\eta_{i}, X_{i+1}+\eta_{i+1}, X_{i+2}+\eta_{i+2})$ for any $i\geq 1$.
The fact that $(X_{i}+\eta_{i})_{i\geq 1}$ is a Markov chain implies that the conditional distribution of $X_{3}+\eta_{3}$, conditionally to $(X_{2}+\eta_{2},X_{1}+\eta_{1})$, equals the conditional distribution of $X_{3}+\eta_{3}$, conditionally to $X_{2}+\eta_{2}$ alone, that is for all real numbers $y_{1}, y_{2}, y_{3}$, 
$$
g_{3}(y_{1},y_{2},y_{3})g_{1}(y_{2})=g_{2}(y_{1},y_{2})g_{2}(y_{2},y_{3}).
$$
This rewrites as follows. 
 For all real numbers $y_{1}, y_{2}, y_{3}$, 
\begin{multline*}
\!\int\! \mu(x_{1})q (x_{1},x_{2})\phi (y_{1}-x_{1})\phi (y_{2}-x_{2})q(x_{2},x_{3})\phi (y_{3}-x_{3})\rmd x_{1}\rmd x_{2}\rmd x_{3} \!\int\! \mu(x_{4})\phi (y_{2}-x_{4})\rmd x_{4}
\\
= \!\int\! \mu(x_{1})q (x_{1},x_{2})\phi (y_{1}-x_{1})\phi (y_{2}-x_{2})\rmd x_{1}\rmd x_{2} \!\int\! \mu(x_{4})q (x_{4},x_{3})\phi (y_{3}-x_{3})\phi (y_{2}-x_{4})\rmd x_{3}\rmd x_{4}.
\end{multline*}
But for all real numbers $x$ and $y$, $\phi (y-x)=\phi (x-y-2m)$.
Since $y$ is a complete statistic for $\phi(x-y-2m) \rmd x$, this 
implies that for all real numbers $x_{1}, x_{3}, y_{2}$, 
\begin{equation}
\label{multigauss2}
\int \mu(x_{1})q (x_{1},x_{2})\mu (x_{4})[q (x_{2},x_{3})-q (x_{4},x_{3})]\phi (y_{2}-x_{2})\phi (y_{2}-x_{4})\rmd x_{2}\rmd x_{4}=0\eqsp.
\end{equation}
Using that $\phi (y_{2}-x_{2})\phi (y_{2}-x_{4})=\phi (\sqrt{2}[y_{2}-(x_{2}+x_{4})/2)]) \phi ((x_{2}-x_{4}+m)/\sqrt{2})$,  (\ref{multigauss2}) implies that for  all real numbers $x_{1}, x_{3}$, $u$,
\begin{equation}
\label{multigauss3}
\int \mu(x_{1})q \left(x_{1},\frac{u+v}{2}\right)\mu \left(\frac{u-v}{2}\right)\left[q \left(\frac{u+v}{2},x_{3}\right)-q \left(\frac{u-v}{2},x_{3}\right)\right]\phi ((v+m)/\sqrt{2})\rmd v=0\eqsp.
\end{equation}
Let $H:\R^{3}\longrightarrow \R$ be any measurable and positive function. Define the measurable and positive function $G: (x,y,z)\mapsto H(x,y,z)\phi ((x-y+2m)/2\sqrt{2})$.
Then by multiplying (\ref{multigauss3}) by $H((u+v)/2, (u-v)/2, x_{3})$ and integrating over $x_{1}, x_{3}$, $u$,
we get by change of variable that 
\begin{multline}
\int \mu(x_{1})q (x_{1},x_{2})q (x_{2},x_{3})\mu (x_{4})G(x_{2},x_{4},x_{3})\rmd x_{1}\rmd x_{2}\rmd x_{3}\rmd x_{4}\\=
\int \mu(x_{1})q (x_{1},x_{2})\mu (x_{4})q (x_{4},x_{3})G(x_{2},x_{4},x_{3})\rmd x_{1}\rmd x_{2}\rmd x_{3}\rmd x_{4}\eqsp.\label{multigauss4}
\end{multline}
Let now $(\tilde{X}_{i})_{i\geq 1}$ be a Markov chain with the same distribution of $(X_{i})_{i\geq 1}$ but independent of  $(X_{i})_{i\geq 1}$. 
Since the correspondance $G \leftrightarrow H$ between measurable positive functions is one-to-one, (\ref{multigauss4}) means that
for any measurable and positive function $G$, 
$
\E\left[ G\left(X_{2},\tilde{X}_{2},X_{3}\right)\right]= \E\left[ G\left(X_{2},\tilde{X}_{2},\tilde{X}_{3}\right)\right]\eqsp,
$
which means that $(X_{2},\tilde{X}_{2},X_{3})$ and $(X_{2},\tilde{X}_{2},\tilde{X}_{3})$ have the same distribution. But this implies that $X_{2}$ is independent of $(\tilde{X}_{2},X_{3})$ which implies that $X_{2}$ is independent of $X_{3}$.

\subsection{Proof of Theorem \ref{theoconsi1}}

Using the fact that characteristic functions are bounded by $1$, for all $R\in \Rcal$,
\begin{equation}
\label{ineqMn}
\left \vert M_{n}(R)-M(R) \right\vert \leq \frac{3}{\sqrt{n} }\sup_{(t_{1},t_{2})\in \Scal} | Z_{n}(t_{1},t_{2}) | + \frac{1}{n }\sup_{(t_{1},t_{2})\in \Scal} | Z_{n}(t_{1},t_{2}) | ^{2}\eqsp,
\end{equation}
and using the assumption on $Z_{n}$, 
$\sup_{R\in \Rcal}\vert M_{n}(R)-M(R)\vert = O_{\po^{\star}}(n^{-1/2})$. Now, using the definition of $\widehat{R}_{n}$ and (\ref{ineqMn}), 
$
M(\widehat{R}_{n})\leq M_{n}(\widehat{R}_{n})+O_{\po^{\star}}(n^{-1/2})
\leq M_{n}(R^{\star})+O_{\po^{\star}}(n^{-1/2})
\leq M(R^{\star})+O_{\po^{\star}}(n^{-1/2})\eqsp.
$
$M(\widehat{R}_{n})$ is then upper bounded by a term of order $O_{\po^{\star}}(n^{-1/2})$ since $M(R^{\star})=0$, and 
the first assertion of Theorem \ref{theoconsi1} is proved. Now, 
$R\mapsto M(R)$ is continuous for the weak convergence topology, and for any $\epsilon >0$, $\sup_{R\in \Rcal, d(R,R^{\star})\geq \epsilon} M(R)$ is attained by compactness of $\{R\in \Rcal, d(R,R^{\star})\geq \epsilon\}$, and positive since $M(R)=0$ if and only if $R=R^{\star}$ up to translation. Thus using  Theorem 5.7 in \cite{MR1652247},
the set of limiting values of $(\widehat{R}_{n})_{n\geq 1}$ for the weak convergence topology is the set of $R\in \Rcal$ such that  $R=R^{\star}$ up to translation.

\section{Proof of Theorem \ref{theoconsi2}}
\label{sec:proof:mle}

\subsection{General statement}
This section provides in Theorem~\ref{theoconsi_mle_general} a more general statement of the result claimed in Theorem~\ref{theoconsi2}. It extends the class of emission densities $\Gamma$ and the models $(G_D)_D$ considered beyond mixtures of Gaussian distributions, but does not change the modelling of the state space.
The proof of Theorem~\ref{theoconsi_mle_general} is postponed to Section~\ref{sec:proof:generaltheorem}.

Let $\Gamma$ be a set of probability densities on $\R^d$ that satisfies the following assumption.
\begin{hypH}
\label{hyp_densiteEmission}
$\Gamma$ is a set of continuous and positive probability densities that admit a first order moment and are centered in the sense that for all $\gamma \in \Gamma$,
\begin{equation}
\label{eq_centering}
\int_{\R^d} y \gamma(y) \rmd y = 0\eqsp.
\end{equation}
$\Gamma$ is a compact subset of $\Lbf^1(\R^d)$ and the envelope function
\begin{equation*}
b : y \in \R^d \longmapsto \sup_{\gamma \in \Gamma} \sup_{x \in \Lambda} \max( \gamma(y-x), \gamma(x-y) )
\end{equation*}
satisfies $b \in \Lbf^1(\R^d) \cap \Lbf^\infty(\R^d)$, admits a first order moment, and there exists a constant $C_\Gamma > 0$ such that for all $\gamma \in \Gamma$ and $y \in \R^d$, the mapping $x \in \Lambda \longmapsto \gamma(y-x) / b(y)$ is $C_\Gamma$-Lipschitz.
Finally, $\gamma^{\star}\in \Gamma$.
\end{hypH}
The centering assumption~\eqref{eq_centering} allows to fix the translation parameter in the identifiability results. 

\begin{ex}
Let $f$ be a bounded and positive probability density on $\R^d$ that admits a first order moment and is centered.
Assume that there exists $\epsilon > 0$ such that
\begin{equation*}
\underset{\|\mu\|_2 \leq \epsilon, \  \|\Sigma - \Id_d\|_F \leq \epsilon}{\sup_{(\mu, \Sigma) \in \R^d \times GL_d(\R)}}
		f(\Sigma(\cdot - \mu))
	\in \Lbf^1(\R^d)
\end{equation*}
and let $\Theta$ be a compact subset of $\R^d \times GL_d(\R)$.
Finally, assume that there exists a function $D_f$ such that for all $y,y' \in \R^d$, $|f(y) - f(y')| \leq D_f(y) |y-y'|$ and such that $(D_f/f) \in \Lbf^\infty(\R^d)$.
Then the set of translation-scale mixtures of $f$ with parameters in $\Theta$
\begin{equation*}
\Gamma = \left\{ \gamma : y \longmapsto \int_\Theta |\det(\Sigma)| f\left( \Sigma(y-\mu) \right) \rmd p(\mu,\Sigma) : p \in \Pcal(\Theta), \int_\Theta \mu \rmd p(\mu, \Sigma) = 0 \right\}
\end{equation*}
satisfies H\ref{hyp_densiteEmission}.
\end{ex}

\begin{hypH}
\label{hyp_tails}
$\Gamma$ satisfies H\ref{hyp_densiteEmission} with the envelope function $b$.
Let $m$ be the lower envelope function of $\Gamma$ defined by
\begin{equation*}
m : y \in \R^d \longmapsto \inf_{\gamma \in \Gamma} \inf_{x \in \Lambda} \gamma(y-x).
\end{equation*}
There exists $\epsilon > 0$ such that $\int b(y) [b(y)/m(y)]^\epsilon \rmd y < \infty$.
\end{hypH}

\begin{ex}
The set $\Gamma$ of Gaussian location-scale mixtures of Section~\ref{sec:mle} satisfies H\ref{hyp_densiteEmission} and H\ref{hyp_tails}.
\end{ex}

Then, consider $(G_D)_{D \geq 1}$  a family of subsets of $\Gamma$. The following assumption essentially means that each $G_D$ is a parametric model with dimension $D$.
\begin{hypH}
\label{hyp_modelesParam}
$\Gamma$ satisfies H\ref{hyp_densiteEmission} and H\ref{hyp_tails} with the functions $b$ and $m$, the set
$\bigcup_{D \geq 1} G_D$ is dense in $\Gamma$ with respect to the $\Lbf^1$ norm, and
there exists a constant $\tilde c > 0$, a mapping $(D,A) \in \N^* \times \R_+ \longmapsto c(D,A)$ and an increasing mapping $D \longmapsto \dim_D$ such that the following holds.
\begin{itemize}
\item For all $D \geq 1$ and $A \geq 0$, $\log c(D,A) \leq \tilde c (\log \dim_D + A)$.
\item For all $D \geq 1$, there exists a surjective application $\theta \in \Theta_D \subset [-1,1]^{\dim_D} \longmapsto \gamma^\theta \in G_D$ such that for all $x \in \Lambda$, $A \geq 0$ and $y \in \R^d$ such that $\log (b(y)/m(y)) \leq A$, the mapping
$
\theta \in \Theta_D \longmapsto \gamma^\theta(y - x) / b(y)
$
is $c(D,A)$-Lipschitz (with $\Theta_D$ endowed with the supremum norm).
\end{itemize}
\end{hypH}
The exact value of $\tilde c$ only matters for the constants in the penalty.

\begin{ex}
The family $(G_D)_{D \geq 1}$ of finite Gaussian translation-scale mixtures defined in Section~\ref{sec:mle} satisfies H\ref{hyp_modelesParam} with $\dim_D = D(d^2 + d) + D-1$ for all $D\geqslant 1$.
\end{ex}
Define the sets $(\Sbf_{r,D})_{r \geq 1, D \geq 1}$, the models $(S_{r,D,n})_{r,D,n}$ and their maximum likelihood estimators $(\hat{\Xfrak}_{r,D,n},\hat{Q}_{r,D,n},\hat{\gamma}_{r,D,n})$ as in Section~\ref{sec:mle}.
Then, select the number of states and the model dimension using the penalized likelihood. Let $\pen(n,r,D)$ be a penalty function such that $\pen(n,r,D) \underset{n \rightarrow +\infty}{\longrightarrow} 0$ for all $r$ and $D$ and such that there exists a sequence $(u_n)_{n \geq 1}$ satisfying $u_n \underset{n \rightarrow \infty}{\longrightarrow} +\infty$ and for all $n$, $r$, $D$,
\begin{equation*}
\pen(n,r,D) \geq u_n (\dim_D + rd + r^2 - 1) \frac{(\log n)^{14} \log \log n}{n}\eqsp.
\end{equation*}
For instance, for any constant $\text{cst} > 0$, this inequality holds by choosing $\pen : (n,r,D) \longmapsto ({\text{cst} \cdot \dim_D} + r^2) \frac{(\log n)^{15}}{n}$.
Let
\begin{equation*}
(\hat{r}_n, \hat{D}_n) \in \argmax_{r \leq \log n, D \text{ s.t. } \dim_D \leq n}
	\left( \frac{1}{n} \ell_n(\hat{\Xfrak}_{r,D,n},\hat{Q}_{r,D,n},\hat{\gamma}_{r,D,n}) - \pen(n,r,D) \right)
\end{equation*}
and define the final estimators
$
(\hat{\Xfrak}_n, \hat{Q}_n, \hat{\gamma}_n) = (\hat{\Xfrak}_{\hat{r}_n, \hat{D}_n, n}, \hat{Q}_{\hat{r}_n, \hat{D}_n, n}, \hat{\gamma}_{\hat{r}_n, \hat{D}_n, n})\eqsp.
$ 

\begin{theo}
\label{theoconsi_mle_general}
Assume that assumptions H\ref{assum::}, H\ref{hyp_transitionKernel}, H\ref{hyp_densiteEmission}, H\ref{hyp_tails} and H\ref{hyp_modelesParam} hold.
Let $\lambda^\star$ be the measure defined in assumption H\ref{hyp_transitionKernel}.
Then, almost surely
\begin{equation*}
\sup_{x \in \Supp(\lambda^\star)} W_1(K_{\hat{\Xfrak}_n,\hat{Q}_n}(x,\cdot),K^{\star}(x, \cdot)) \underset{n \rightarrow \infty}{\longrightarrow} 0
\end{equation*}
and
$
\| \hat{\gamma}_n - \gamma^\star \|_1 \underset{n \rightarrow \infty}{\longrightarrow} 0
$.
In particular, almost surely under $\po^\star$, for all $x \in \Supp(\lambda^\star)$,
$
{ K_{\hat{\Xfrak}_n,\hat{Q}_n}(x,\cdot) 
\longrightarrow 
K^{\star}(x, \cdot) }
$
for the weak convergence topology and if $\po^X_{K}$ denotes the distribution of the stationary Markov chain with transition kernel $K$,
$
{ \po^X_{K_{\hat{\Xfrak}_n,\hat{Q}_n}} 
\longrightarrow
\po^X_{K^{\star}} }
$
for the weak convergence topology.
\end{theo}
The remaining sections of this paper are dedicated to the proof of Theorem~\ref{theoconsi_mle_general}.

\subsection{Proof of Theorem~\ref{theoconsi_mle_general}}
\label{sec:proof:generaltheorem}

This section  states a few intermediate results whose proofs are  postponed to the following sections. These results are followed by the proof of Theorem~\ref{theoconsi_mle_general},  the consistency of the maximum likelihood estimator, which is the main result of this appendix.
Let $\Omega_\omega^C$ be the set of transition kernels $K$ on $\Lambda$ which admit the modulus of continuity $\omega$ with respect to the Wasserstein 1 metric and such that there exists a probability measure $\lambda$ (which may depend on $K$) such that for all $x \in \Lambda$, $K(x,\cdot)$ is absolutely continuous with respect to $\lambda$ with a density taking values in $[1/C,C]$. The kernel $K^\star$ as well as all kernels considered in the models $S_{r,D}$ belong to $\Omega_\omega^C$.

\begin{lem}
\label{lem_compactness}
Assume that $\Omega_\omega^C$ is endowed with the topology of the uniform convergence on the set of continuous functions with values in $(\Pcal(\Lambda), W_1)$, and $\Gamma$ is endowed with the $\Lbf_1$ topology. Then $\Omega_\omega^C \times \Gamma$ endowed with the product topology is compact.
\end{lem}
For all probability measures $\mu$ and $\nu$, the Kullback Leibler divergence between $\mu$ and $\nu$ is defined by
\begin{equation*}
KL(\mu \| \nu) = \begin{cases}
\int \log \frac{\text{d} \mu}{\text{d} \nu} \text{d}\mu &\text{when $\mu$ is absolutely continuous with respect to $\nu$}, \\
+\infty &\text{otherwise}.
\end{cases}
\end{equation*}

\begin{lem}
\label{lem_KL_equivalent_VT}
Let $(K_n, \gamma_n)_{n \geq 1} \in (\Omega_\omega^C \times \Gamma)^{\N^*}$. For all $n \geq 1$, the quantity
$
\Kbf(\po_{K^\star,\gamma^\star} \| \po_{K_n,\gamma_n})
	= \lim_{m \rightarrow +\infty} \frac{1}{m} KL(\po^{(m)}_{K^\star,\gamma^\star} \| \po^{(m)}_{K_n,\gamma_n})
$
exists and is finite, and the following two statements are equivalent.
\begin{enumerate}
\item $\Kbf(\po_{K^\star,\gamma^\star} \| \po_{K_n,\gamma_n}) \underset{n \rightarrow \infty}{\longrightarrow} 0$.
\item For all $k \geq 1$, $d_{TV}(\po^{(k)}_{K^\star,\gamma^\star}, \po^{(k)}_{K_n,\gamma_n}) \underset{n \rightarrow \infty}{\longrightarrow} 0$.
\end{enumerate}
\end{lem}
The consistency of the maximum likelihood estimator relies on the following oracle inequality, which follows from \cite[Theorem~8]{lehericy2018misspe}. It is proved in detail in Section~\ref{sec_full_oracle_inequality} how Proposition~\ref{prop_oracle_simplifie} is deduced from \cite[Theorem~8]{lehericy2018misspe} in the setting of this paper.
\begin{prop}

\label{prop_oracle_simplifie}
For each $r, D$ and $n$, let $S_{r,D,n}$ and $(\hat{\Xfrak}_{n}, \hat{Q}_{n}, \hat{\gamma}_{n})$ be defined as in Section~\ref{sec:mle}. There exist constants $C_\pen$, $A$ and $n_0$ such that the following holds. Assume that the penalty satisfies $\pen(n,r,D) \geq C_\pen (\dim_D + rd + r^2 - 1) \log(n)^{14} / n$ for all $n \geq n_0$, $r$ and $D$. Then, for all $n \geq n_0$, with probability at least $1 - 3n^{-2}$,
\begin{multline*}
\Kbf(\po_{K^\star,\gamma^\star} \| \po_{\hat{\Xfrak}_n, \hat{Q}_n, \hat{\gamma}_n})
	\\ \leq 2 \inf_{r \leq \log n, D \text{ s.t.} \dim_D \leq n} \left( 
		\inf_{(\Xcal, Q, \gamma) \in S_{r,D}} \Kbf(\po_{K^\star,\gamma^\star} \| \po_{\Xfrak, Q, \gamma})
		+ 2\pen(n,r,D)
	\right) 
	+ A \frac{(\log n)^9}{n}\eqsp.
\end{multline*}
\end{prop}

\begin{lem}
\label{lem_VT_continue}
Let $(K_n, \gamma_n)_{n \geq 1} \in (\Omega_\omega^C \times \Gamma)^{\N^*}$ be a sequence that converges to $(K,\gamma)$. Then, for all $k \geq 1$, $d_{TV}(\po^{(k)}_{K,\gamma}, \po^{(k)}_{K_n,\gamma_n}) \underset{n \rightarrow \infty}{\longrightarrow} 0$.
\end{lem}

\begin{lem}
\label{lem_approx_noyau_existe}
There exists a sequence $(\Xfrak_t, Q_t, \gamma_t)_{t \geq 1}$ taking values in $\bigcup_{r \geq 1, D \geq 1} \Sbf_{r,D}$ such that
$\Kbf(\po_{K^\star,\gamma^\star} \| \po_{\Xfrak_t,Q_t,\gamma_t}) \underset{t \rightarrow \infty}{\longrightarrow} 0$
and
$R_{K_{\Xfrak_t,Q_t}} \underset{t \rightarrow \infty}{\longrightarrow} R_{K^\star}$
in distribution.
\end{lem}
Let us now discuss the choice of $\majZ'$, $(\majZ_n)_n$ and $\underline{\varphi}$ in equation~\eqref{eq_strong_H1_seq}. Let $\majZ'$ be a positive real number. $(\majZ_n)_n$ and $\underline{\varphi}$ are chosen such that there exists sequences $r_n, t_n \rightarrow +\infty$ with $r_n \leq \log n$ for $n$ large enough such that $(\Xfrak_{t_n}, Q_{t_n}, \gamma_{t_n}) \in S_{r_n,D,n}$ for all $n$ (the choice of $D$ does not matter since $(\gamma_t)_t$ can be replaced by any sequence that converges to $\gamma^\star$). Let us show that such a choice is possible. Let $t_n \rightarrow \infty$ and $(r_n)_n$ be such that $r_n \leq \log n$ and $(\Xfrak_{t_n},Q_{t_n},\gamma_{t_n}) \in \bigcup_{D \geq 1} \Sbf_{r_n,D}$ for all $n$ large enough. With the notation $\Mcal_{\rho,d'}(A,B)$ defined in Section~\ref{sec:least:squares}, the sequence $(R_{K_{\Xfrak_t,Q_t}})_{t \geq 1}$ takes values in $\Mcal_{1,2d}(1,L\sqrt{2d})$. By equicontinuity of $\{\Phi_R : R \in \Mcal_{1,2d}(1,L\sqrt{2d})\}$, the convergence
\begin{equation*}
\Phi_{R_{K_{\Xfrak_t,Q_t}}} \underset{t \rightarrow \infty}{\longrightarrow} \Phi_{R_{K^\star}}
\end{equation*}
holds uniformly over all compacts of $\C^{2d}$. Take $\underline{\varphi}$ such that equation~\eqref{eq_strong_H1} holds with $\majZ^{eq.\eqref{eq_strong_H1}} = \majZ'$ and $\underline{\varphi}^{eq.\eqref{eq_strong_H1}} = 2 \underline{\varphi}$ since $R_{K^\star}$ satisfies equation~\eqref{eq_strong_H1}. Then the uniform convergence over all compacts entails that by choosing $\majZ_n \rightarrow \infty$ slowly enough, the desired property holds.

Theorem~\ref{theoconsi_mle_general} may now be proved. Proposition~\ref{prop_oracle_simplifie} actually gives a deterministic function $f : \N^* \longrightarrow \R_+$ such that for all $n \geq n_0$, with probability at least $1 - 3n^{-2}$,
\begin{equation*}
\Kbf(\po_{K^\star,\gamma^\star} \| \po_{\hat{\Xfrak}_n, \hat{Q}_n, \hat{\gamma}_n}) \leq f(n)\eqsp.
\end{equation*}
By the previous paragraph and the assumption that $\pen(n,r,D)$ goes to zero as $n$ goes to infinity for each $r$ and $D$, $f \rightarrow 0$. Hence, by Borel-Cantelli's Lemma, almost surely,
\begin{equation*}
\Kbf(\po_{K^\star,\gamma^\star} \| \po_{\hat{\Xfrak}_n, \hat{Q}_n, \hat{\gamma}_n}) \underset{n \rightarrow \infty}{\longrightarrow} 0\eqsp.
\end{equation*}
Thus, by Lemma~\ref{lem_KL_equivalent_VT}, almost surely, for all $k \geq 1$,
\begin{equation*}
d_\text{TV} \left(\po^{(k)}_{K^\star,\gamma^\star}, \po^{(k)}_{\hat{\Xfrak}_n, \hat{Q}_n, \hat{\gamma}_n}\right)
	\underset{n \rightarrow +\infty}{\longrightarrow} 0\eqsp.
\end{equation*}
In particular, by Lemma~\ref{lem_VT_continue}, all limits $(K,\gamma)$ of convergent subsequences of $(K_{\hat{\Xfrak}_n, \hat{Q}_n}, \hat{\gamma}_n)_n$ satisfy $\po^{(2)}_{K^\star,\gamma^\star} = \po^{(2)}_{K,\gamma}$. Since the support of $X$ is in the known compact set $\Lambda$, $\mu_K \in \Mcal_1$. Moreover, equation~\eqref{eq_strong_H1_seq} entails that $K$ satisfies H\ref{assum::}.
Since the translation parameter is fixed by the centering condition on the densities, Theorem~\ref{theoident1} ensures that $R_{K^\star} = R_K$ and $\gamma = \gamma^\star$. Therefore, using the continuity of $K$ and $K^\star$, it follows that $K(x,\cdot) = K^\star(x,\cdot)$ for all $x \in \Supp(\lambda^\star)$.
Since the set of parameters is compact by Lemma~\ref{lem_compactness}, the estimators converge to the true parameters, which is the first part of Theorem~\ref{theoconsi_mle_general}. Finally, since $K^\star$ admits a unique stationary distribution, Theorem 4 and the corollary of Theorem 6 of \cite{karr75} entail that
\begin{equation*}
\po^X_{K_{\hat{\Xfrak}_n, \hat{Q}_n}} \overset{(d)}{\underset{n \rightarrow \infty}{\longrightarrow}} \po^X_{K^\star}\eqsp,
\end{equation*}
which concludes the proof of Theorem~\ref{theoconsi_mle_general}.

\subsection{Proof of Lemma~\ref{lem_compactness}}

Let $\Omega_\omega$ be the set of transition kernels on $\Lambda$ which admit the modulus of continuity $\omega$ with respect to the Wasserstein 1 metric. $\Omega_\omega$ is an equicontinuous family of functions from $\Lambda$ to the set of probability measures $\Pcal(\Lambda)$ on $\Lambda$ endowed with the Wasserstein 1 metric. Since $\Lambda$ is compact, convergence in Wasserstein distance is equivalent to convergence in distribution and $\Pcal(\Lambda)$ is compact for the topology of the convergence in distribution, so that Arzelà-Ascoli's theorem ensures that $\Omega_\omega$ is relatively compact in the class of continuous functions from $\Lambda$ to $(\Pcal(\Lambda), W_1)$ with respect to the uniform convergence distance. It is closed, therefore it is compact.

Recall that $\Omega_\omega^C$ is the subset of $\Omega_\omega$ such that $K \in \Omega_\omega^C$ if and only if there exists a probability measure $\lambda$ such that for all $x \in \Lambda$, $K(x,\cdot)$ is absolutely continuous with respect to $\lambda$ with a density taking values in $[1/C,C]$. Let us show that it is closed. Let $(K_n)_{n \geq 1}$ be a convergent sequence in $\Omega_\omega^C$ and $(\lambda_n)_{n \geq 1}$ the associated probability measures. Write $K \in \Omega_\omega$ its limit. Without loss of generality, it is possible to assume that $\lambda_n \longrightarrow_{} \lambda$ for some $\lambda \in \Pcal(\Lambda)$ as $n$ grows to $+\infty$. Let $\Ccal^0_{b,+}$ be the set of real-valued, nonnegative, bounded and continuous function on $\Lambda$, then for all $f \in \Ccal^0_{b,+}$ and all $x \in \Lambda$,
\begin{align*}
\int K_n(x,\rmd x') f(x') \in \left[\frac{1}{C}\int f \rmd\lambda,C\int f \rmd\lambda\right]
\end{align*}
by definition of $\Omega_\omega^C$. Then, using the convergence of the sequences, for all $f \in \Ccal^0_{b,+}$ and all $x \in \Lambda$,
\begin{align*}
\int K(x,\rmd x') f(x') \in \left[\frac{1}{C}\int f \rmd\lambda,C\int f \rmd\lambda\right] \eqsp.
\end{align*}
For all closed set $F \subset \Lambda$, there exists a sequence $(f_i)_{i \geq 1} \searrow \one_F$. Therefore, for all closed set $F \subset \Lambda$ and all $x \in \Lambda$,
\begin{align*}
K(x,F) \in \left[\frac{\lambda(F)}{C},C\lambda(F)\right] \eqsp.
\end{align*}
Thus, using the regularity of Borel probability measures on polish spaces, the same holds for all measurable sets, so that $K \in \Omega_\omega^C$. Therefore, $\Omega_\omega^C$ is closed, so that it is compact.

\subsection{Proof of Lemma~\ref{lem_KL_equivalent_VT}}
\label{preuve_lem_KL_equivalent_VT}

The following lemma follows from the proof of Lemma 3 of \cite{douc:moulines:ryden:04}. In this section only, for all integers $a \leq b$, write $Y_a^b$ instead of $(Y_a, \dots, Y_b)$.
\begin{lem}
\label{lem_douc04_revu}
Assume that assumption H\ref{hyp_densiteEmission} holds. By stationarity, extend the process $(Y_t)_{t \geq 1}$ into a process $(Y_t)_{t \in \Zbb}$. Let $K, K' \in \Omega_\omega^C$ and $\gamma, \gamma' \in \Gamma$.
Then, there exists random variables $\delta_{k, \infty}(K,\gamma)$ and $\delta_{k, \infty}(K',\gamma')$ such that almost surely, for all $k \in \Zbb$ and $m \geq 0$,
\begin{equation*}
\left|
	\log \frac{p_{Y_{k} | Y_{k-m}^{k-1}, K,\gamma}(Y_{k} | Y_{k-m}^{k-1})}{p_{Y_{k} | Y_{k-m}^{k-1}, K',\gamma'}(Y_{k} | Y_{k-m}^{k-1})}
	- \log \frac{\delta_{k, \infty}(K,\gamma)}{\delta_{k, \infty}(K',\gamma')}
\right| \leq 2 C^2 \left(1 - \frac{1}{C^2}\right)^{m-1}\eqsp,
\end{equation*}
and for all $k \in \Zbb$,
\begin{equation*}
\left(\sup_{m \geq 0} \left| \log \frac{p_{Y_{k} | Y_{k-m}^{k-1}, K,\gamma}(Y_{k} | Y_{k-m}^{k-1})}{p_{Y_{k} | Y_{k-m}^{k-1}, K',\gamma'}(Y_{k} | Y_{k-m}^{k-1})} \right| \right)
	\vee \left| \log \frac{\delta_{k, \infty}(K,\gamma)}{\delta_{k, \infty}(K',\gamma')} \right|
	\in \Lbf^1(\po^\star)\eqsp.
\end{equation*}
\end{lem}

\begin{proof}
Write first how the notations of this paper match those of \cite{douc:moulines:ryden:04}. The set $\Xcal$ (resp. $\mathcal{Y}$ ) of \cite{douc:moulines:ryden:04} is $\Lambda$ (resp. $\R^d$) and $\R^d$  is equiped with the measure with density $b/\|b\|_1$ with respect to the Lebesgue measure. Finally, the set $\Theta$ of \cite{douc:moulines:ryden:04} is $\{ (K, \gamma), (K', \gamma') \}$. Contrary to the setting of \cite{douc:moulines:ryden:04}, $\Xcal$ is endowed  with a measure that depends on the parameter $\theta$. The proof of Lemma 3 of \cite{douc:moulines:ryden:04} holds with the following relaxed assumptions (with the notations of \cite{douc:moulines:ryden:04}).
\begin{description}
\item[(A1')] For all $\theta \in \Theta$, there exists a measure $\mu_\theta$ on $\Xcal$ such that the transition kernel of $(X_k)_{k \geq 1}$ has a density $q_\theta$ with respect to $\mu_{\theta}$ such that for all $x, x' \in \Xcal$, $1/C \leq q_\theta(x,x') \leq C$.
\item[(A3')] $\bar{\E}_{\theta^*} [ | \log b_+(Y_1, \bar{\mathbf{Y}}_0) | ] < \infty$ and $\bar{\E}_{\theta^*} [ | \log b_-(Y_1, \bar{\mathbf{Y}}_0) | ] < \infty$ where
\begin{align*}
b_+(y_1, \bar{\mathbf{y}}_0) &\overset{\Delta}{=} \sup_{\theta} \int_{\Xcal} g_\theta(y_1 | \bar{\mathbf{y}}_0, x) \mu_\theta(\rmd x), \\
b_-(y_1, \bar{\mathbf{y}}_0) &\overset{\Delta}{=} \inf_{\theta} \int_{\Xcal} g_\theta(y_1 | \bar{\mathbf{y}}_0, x) \mu_\theta(\rmd x).
\end{align*}
\end{description}
These assumptions are equivalent to the following (A1'') and (A3'').
\begin{description}
\item[(A1'')] There exists a measure $\lambda_K$ on $\Lambda$ such that the transition kernel $K$ has a density with respect to $\lambda_K$ with values in $[1/C, C]$, and likewise for $K'$.
\item[(A3'')]  $\E^\star [ | \log \int_{\Lambda} \|b\|_1 (\gamma(Y_1 - x) / b(Y_1)) \rmd\lambda_K(x) | ] < \infty$, and likewise for $(K',\gamma')$.
\end{description}
The lemma then follows from Lemma 3 of \cite{douc:moulines:ryden:04} applied on $(K,\gamma)$ and $(K',\gamma')$. (A1'') is direct by definition of $\Omega_\omega^C$.  By H\ref{hyp_tails}, $\|b\|_1 m(y) / b(y) \leq \int_{\Lambda} g_x(y) d\lambda_K(x) \leq \|b\|_1$. Thus, (A3'') is implied by the integrability condition of H\ref{hyp_tails} since the distribution of $Y_1$ under $\po^\star$ is dominated by the distribution with density $b$ with respect to the Lebesgue measure.
\end{proof}
Thus, for all $K, K' \in \Omega_\omega^C$ and $\gamma, \gamma' \in \Gamma$, the limit
\begin{align*}
\Kbf(\po_{K,\gamma} \| \po_{K',\gamma'})
	= \lim_{m \rightarrow +\infty} \frac{1}{m} KL(\po^{(m)}_{K,\gamma} \| \po^{(m)}_{K',\gamma'})
	= \E_{K,\gamma} \left[ \log \frac{\delta_{0,\infty}(K,\gamma)}{\delta_{0,\infty}(K',\gamma')} \right]
\end{align*}
exists, is finite, and for all $k,m \geq 1$,
\begin{equation*}
\left| k \Kbf(\po_{K,\gamma} \| \po_{K',\gamma'}) - 
	\left(KL(\po^{(m+k)}_{K,\gamma} \| \po^{(m+k)}_{K',\gamma'})
	- KL(\po^{(m)}_{K,\gamma} \| \po^{(m)}_{K',\gamma'}) \right) \right|
	\leq 2 C^4 \left( 1 - \frac{1}{C^2} \right)^{m-1}\eqsp.
\end{equation*}
Let $(K_n, \gamma_n)_{n \geq 1} \in (\Omega_\omega^C \times \Gamma)^\N$ be a sequence of parameters such that $\Kbf(\po_{K^\star,\gamma^\star} \| \po_{K_n,\gamma_n}) \longrightarrow 0$.
The above equation implies that for all $k \geq 1$, there exists sequences $(m_n)_{n \geq 1} \longrightarrow +\infty$ and $(l_n)_{n \geq 1} \longrightarrow +\infty$ such that
\begin{equation*}
KL(\po^{(m_n+l_n+k)}_{K^\star,\gamma^\star} \| \po^{(m_n+l_n+k)}_{K_n,\gamma_n})
	- KL(\po^{(m_n)}_{K^\star,\gamma^\star} \| \po^{(m_n)}_{K_n,\gamma_n}) 
	\underset{n \rightarrow \infty}{\longrightarrow} 0\eqsp.
\end{equation*}
Using the chain rule and Pinsker's inequality,
\begin{align*}
KL(\po^{(m_n+l_n+k)}_{K^\star,\gamma^\star} & \| \po^{(m_n+l_n+k)}_{K_n,\gamma_n})
	- KL(\po^{(m_n)}_{K^\star,\gamma^\star} \| \po^{(m_n)}_{K_n,\gamma_n}) \\
	&= \E_{Y_1^{m_n} | K^\star,\gamma^\star} \left[
		KL\left(\po_{Y_{m_n+1}^{m_n+l_n+k} | Y_1^{m_n}, K^\star,\gamma^\star} \| \po_{Y_{m_n+1}^{m_n+l_n+k} | Y_1^{m_n}, K_n,\gamma_n}\right) 
	\right]\eqsp, \\
	&\geq \E_{Y_1^{m_n} | K^\star,\gamma^\star} \left[
		KL\left(\po_{Y_{m_n+l_n+1}^{m_n+l_n+k} | Y_1^{m_n}, K^\star,\gamma^\star} \| \po_{Y_{m_n+l_n+1}^{m_n+l_n+k} | Y_1^{m_n}, K_n,\gamma_n}\right) 
	\right]\eqsp, \\
	&\geq 2 \E_{Y_1^{m_n} | K^\star,\gamma^\star} \left[
		d_\text{TV}^2 \left(\po_{Y_{m_n+l_n+1}^{m_n+l_n+k} | Y_1^{m_n}, K^\star,\gamma^\star}, \po_{Y_{m_n+l_n+1}^{m_n+l_n+k} | Y_1^{m_n}, K_n,\gamma_n}\right) 
	\right]\eqsp.
\end{align*}
Since the kernels satisfy the Doeblin condition (see for instance \cite{cappe:ryden:2004}, Section 4.3.3), the resulting processes are $\phi$-mixing with mixing coefficients $\phi(i) \leq 2(1 - 1/C)^i$ (see the proof of Lemma 1 of \cite{lehericy2018misspe} for a proof, and \cite{bradley2005strongmixingsurvey} for a survey of mixing properties). In particular, for all $K \in \Omega_\omega^C$, for all positive and continuous probability density $\gamma$ on $\R^d$ and for all $A \in \sigma(Y_1, \dots, Y_{m_n})$ such that $\po_{K, \gamma}(A) > 0$,
\begin{equation*}
d_\text{TV} \left(\po_{Y_{m_n+l_n+1}^{m_n+l_n+k} | A, K,\gamma}, \po_{Y_{m_n+l_n+1}^{m_n+l_n+k} | K,\gamma}\right) 
	\leq 2 \left( 1 - \frac{1}{C} \right)^{l_n},
\end{equation*}
so that using the continuity and positivity of $\gamma$, 
\begin{equation*}
d_\text{TV} \left(\po_{Y_{m_n+l_n+1}^{m_n+l_n+k} | Y_1^{m_n}, K,\gamma}, \po_{Y_{m_n+l_n+1}^{m_n+l_n+k} | K,\gamma}\right) 
	\leq 2 \left( 1 - \frac{1}{C} \right)^{l_n}\eqsp.
\end{equation*}
Finally,
\begin{align*}
2 \E_{Y_1^{m_n} | K^\star,\gamma^\star} & \left[
		d_\text{TV}^2 \left(\po_{Y_{m_n+l_n+1}^{m_n+l_n+k} | Y_1^{m_n}, K^\star,\gamma^\star}, \po_{Y_{m_n+l_n+1}^{m_n+l_n+k} | Y_1^{m_n}, K_n,\gamma_n}\right) 
	\right] \\
	&\geq 2 \left(
		d_\text{TV} \left(\po_{Y_{m_n+l_n+1}^{m_n+l_n+k} | K^\star,\gamma^\star}, \po_{Y_{m_n+l_n+1}^{m_n+l_n+k} | K_n,\gamma_n}\right)
		- 4 \left( 1 - \frac{1}{C} \right)^{l_n}
		\right)^2\eqsp, \\
	&\geq d_\text{TV}^2 \left(\po^{(k)}_{K^\star,\gamma^\star}, \po^{(k)}_{K_n,\gamma_n}\right) - 32 \left( 1 - \frac{1}{C} \right)^{2l_n}\eqsp,
\end{align*}
using that $(a-b)^2 \geq a^2/2 - b^2$ for all $a,b \in \R$ and the stationarity of the distributions $\po_{K,\gamma}$ for all $K \in \Omega_\omega^C$ and $\gamma \in \Gamma$. Therefore, for all $k \geq 1$,
\begin{equation*}
d_\text{TV} \left(\po^{(k)}_{K^\star,\gamma^\star}, \po^{(k)}_{K_n,\gamma_n}\right)
	\underset{n \rightarrow +\infty}{\longrightarrow} 0\eqsp.
\end{equation*}
Conversely, let $(K_n, \gamma_n)_{n \geq 1} \in (\Omega_\omega^C \times \Gamma)^{\N^*}$ be a sequence of parameters such that for all $k \geq 1$, 
$$
d_\text{TV} \left(\po^{(k)}_{K^\star,\gamma^\star}, \po^{(k)}_{K_n,\gamma_n}\right) \underset{n \rightarrow +\infty}{\longrightarrow} 0\eqsp.
$$
Then by Lemma~\ref{lem_douc04_revu}, for all $k, n \geq 1$,
\begin{align}
\nonumber
\Kbf(\po_{K^\star,\gamma^\star} \| \po_{K_n, \gamma_n})
	&\leq \E KL(\po_{Y_k | Y_1^{k-1}, K^\star,\gamma^\star} \| \po_{Y_k | Y_1^{k-1}, K_n, \gamma_n}) + 2C^2 \left(1 - \frac{1}{C^2}\right)^{k-2} \\
	\label{eq_approx_forgetting}
	&\leq KL(\po^{(k)}_{K^\star,\gamma^\star} \| \po^{(k)}_{K_n, \gamma_n}) + 2C^2 \left(1 - \frac{1}{C^2}\right)^{k-2}\eqsp,
\end{align}
by the entropy chain rule.
Lemma 4 of \cite{STG13bayesiandensity} entails that there exists $\lambda_0 \in (0,1)$ such that for all $\lambda \in (0, \lambda_0)$,
\begin{align*}
KL(\po^{(k)}_{K^\star,\gamma^\star} \| \po^{(k)}_{K_n, \gamma_n})
	&\leq \left(1 + 2 k \log \frac{1}{\lambda} \right) h^2(\po^{(k)}_{K^\star,\gamma^\star}, \po^{(k)}_{K_n, \gamma_n}) \\
	&\qquad	+ 2\E\left[ \log \left( \frac{p_{Y_1^k | K^\star,\gamma^\star}}{p_{Y_1^k | K_n,\gamma_n}} \right) \one\left( \frac{p_{Y_1^k | K^\star,\gamma^\star}}{p_{Y_1^k | K_n,\gamma_n}} \geq \frac{1}{\lambda} \right) \right] \eqsp, \\
	&\leq 2 \left(1 + 2 k \log \frac{1}{\lambda} \right) d_\text{TV}( \po^{(k)}_{K^\star, \gamma^\star,}, \po^{(k)}_{K_n, \gamma_n}) \\
	&\qquad	+ 2 \int \prod_{i=1}^k b(y_i) \log \left( \prod_{i=1}^k \frac{b(y_i)}{m(y_i)} \right) \one\left( \prod_{i=1}^k \frac{b(y_i)}{m(y_i)} \geq \frac{1}{\lambda} \right) \rmd y \eqsp,
\end{align*}
using that the square of the Hellinger distance is upper bounded by the $\Lbf^1$ distance, that is twice the total variation distance. The second term is finite for all $\lambda$ by H\ref{hyp_tails}. Therefore, by carefully choosing a sequence $\lambda$ that tends to zero, we obtain $\limsup_{n} KL(\po^{(k)}_{K^\star,\gamma^\star} \| \po^{(k)}_{K_n, \gamma_n}) = 0$ for all $k \geq 1$.
This, together with taking $k$ that tends to infinity in Equation~\eqref{eq_approx_forgetting}, proves the second statement of the lemma.

\subsection{Proof of Lemma~\ref{lem_VT_continue}}

The set of possible parameters $\Omega_\omega^C \times \Gamma$ is endowed with the product topology induced by the uniform convergence topology on $\Omega_\omega^C$ and the $\Lbf^1$ norm on $\Gamma$. It is compact for this topology.
Let $(K_n, \gamma_n)_{n \geq 1}$ be a sequence in $\Omega_\omega \times \Gamma$ that converges to $(K, \gamma)$ with respect to this topology. The aim is now to show that the distribution of $(Y_1, \dots, Y_k)$ with parameters $(K_n, \gamma_n)$ converges in total variation distance to the distribution with parameters $(K,\gamma)$. The transition kernel $K$ admits a unique stationary distribution, so that Theorem 4 and the corollary of Theorem 6 of \cite{karr75} entail that
\begin{equation}
\label{eq_cvg_loi_loi_X}
\po^X_{K_n} \overset{(d)}{\underset{n \rightarrow \infty}{\longrightarrow}} \po^X_K\eqsp,
\end{equation}
where $\po^X_{K}$ denotes the distribution of a stationary Markov chain $(X_n)_{n \geq 1}$ with transition kernel $K$. This convergence holds for the distribution of the whole Markov chain, which implies in particular that the distribution of $k$-tuples $(X_1, \dots, X_k)$ for all $k \geq 1$ converges in the same way. For any $k \geq 1$, the total variation distance between the distributions of $(Y_1, \dots, Y_k)$ is, up to a factor 2,
\begin{align*}
\| p_{(Y_1, \dots, Y_k) | K, \gamma} - p_{(Y_1, \dots, Y_k) | K_n, \gamma_n} \|_1 & = \! \int \!\left|
			\int \prod_{i=1}^k \gamma(y_i - x_i) \rmd\po^X_{K}(x)
			\!-\! \int \prod_{i=1}^k \gamma_n(y_i - x_i) \rmd\po^X_{K_n}(x)
		\right| \rmd y\eqsp, \\
	& \!\leq \int \left|
			\int \prod_{i=1}^k \gamma(y_i - x_i) \rmd\po^X_{K}(x)
			- \int \prod_{i=1}^k \gamma(y_i - x_i) \rmd\po^X_{K_n}(x)
		\right| \rmd y\eqsp, \\
		&\qquad \qquad + \int \int \left| \prod_{i=1}^k \gamma(y_i - x_i) - \prod_{i=1}^k \gamma_n(y_i - x_i) \right| \rmd\po^X_{K_n}(x) \rmd y \eqsp.
\end{align*}
Consider the first term of the right hand side. Since $x \longmapsto \gamma(y-x)$ is continuous and bounded for all $y \in \R^d$, Equation~\eqref{eq_cvg_loi_loi_X} yields, for all $y \in \R^d$, 
\begin{equation*}
\left|
	\int \prod_{i=1}^k \gamma(y_i - x_i) \rmd\po^X_{K}(x)
	- \int \prod_{i=1}^k \gamma(y_i - x_i) \rmd\po^X_{K_n}(x)
\right| \underset{n \rightarrow \infty}{\longrightarrow} 0\eqsp.
\end{equation*}
Then, since $\sup_{x \in \Lambda} \gamma(y-x) \leq b(y)$ for all $y \in \R^d$,
$
\left| \int \prod_{i=1}^k \gamma(y_i - x_i) \rmd\po^X_{K}(x) \right|
	\leq \prod_{i=1}^k b(y_i)\eqsp,
$
and the right hand side  is integrable. The same holds for $K_n$, so that the dominated convergence theorem implies
\begin{equation*}
\int \left|
	\int \prod_{i=1}^k \gamma(y_i - x_i) \rmd\po^X_{K}(x)
	- \int \prod_{i=1}^k \gamma(y_i - x_i) \rmd\po^X_{K_n}(x)
\right| \rmd y
	\underset{n \rightarrow \infty}{\longrightarrow} 0\eqsp.
\end{equation*}
For the second term, write
\begin{align*}
\int \int & \left| \prod_{i=1}^k \gamma(y_i - x_i) - \prod_{i=1}^k \gamma_n(y_i - x_i)\right| \rmd\po^X_{K_n}(x) \rmd y \\
	&\leq \sum_{i=1}^k \int \int \prod_{j < i} \gamma(y_j - x_j) \left| \gamma(y_i - x_i) - \gamma_n(y_i - x_i) \right| \prod_{j > i} \gamma_n(y_j - x_j) \rmd\po^X_{K_n}(x) \rmd y \eqsp, \\
	&\leq \sum_{i=1}^k \int \int |\gamma(y_i - x_i) - \gamma_n(y_i - x_i)| \rmd y_i
	\rmd\po^X_{K_n}(x_i) \eqsp, \\
	&= k \|\gamma - \gamma_n\|_1\eqsp,
\end{align*}
where the last term converges to 0 as $n\rightarrow \infty$.
Hence, $d_\text{TV}( \po^{(k)}_{K, \gamma}, \po^{(k)}_{K_n, \gamma_n}) \underset{n \rightarrow \infty}{\longrightarrow} 0$ for all $k \geq 1$.

\subsection{Proof of Lemma~\ref{lem_approx_noyau_existe}}
By Lemmas~\ref{lem_KL_equivalent_VT} and~\ref{lem_VT_continue}, to show the convergence with $\Kbf$, it suffices to show that there exists a sequence $(\Xfrak_t,Q_t)_{t \geq 1}$ such that $(\Xfrak_t,Q_t, -) \in \bigcup_{r, D} S_{r,D}$ and such that the sequence of kernels $(K_t)_{t \geq 1} = (K_{\Xfrak_t,Q_t})_{t \geq 1}$ converges to $K^\star$. It also suffices to show the convergence of $(R_{K_{\Xfrak_t,Q_t}})_{t \geq 1}$ since $K^\star$ admits a unique stationary distribution by using Theorem 4 and the corollary of Theorem 6 of \cite{karr75}.

The following lemma, which is a consequence of simple algebra, is stated without proof.

\begin{lem}
\label{lemma_discretisation_support}
Let $\lambda$ be a probability measure on a compact set of $\R^d$ which is absolutely continuous with respect to the Lebesgue measure. Then, there exists a sequence of integers $(r_t)_{t \geq 1} \longrightarrow +\infty$ and a sequence $((A_i^t)_{1 \leq i \leq r_t})_{t \geq 1}$ of measurable partitions of the support of $\lambda$ such that
\begin{equation*}
\begin{cases}
D_t = \underset{1 \leq i \leq r_t}{\sup} \text{diam}(A_i^t) \underset{t \rightarrow +\infty}{\longrightarrow} 0\eqsp, \\
\forall t \geq 1, \quad \forall 1 \leq i \leq r_t, \quad
	\lambda(A_i^t) \in \left[ \frac{1}{2 r_t}, \frac{2}{r_t} \right]\eqsp.
\end{cases}
\end{equation*}
\end{lem}
To address the case where $\lambda^\star$ is not absolutely continuous with respect to the Lebesgue measure, consider convolutions of the kernels. For all $\epsilon \in (0,1]$, let $U_\epsilon$ be the uniform measure on $[-\epsilon,\epsilon]^d$. For all probability measure $\lambda$ on $\R^d$, write $\lambda * U_\epsilon$ the convolution of $\lambda$ and $U_\epsilon$, and for all transition kernel $K$ on $\R^d$, write $K * U_\epsilon$ the transition kernel defined by $(K * U_\epsilon)(x,\cdot) = K(x,\cdot) * U_\epsilon$.
Then $K^\star * U_\epsilon$ admit the modulus of continuity $\omega$ for all $\epsilon > 0$ (since $W_1(\mu * U_\epsilon, \nu * U_\epsilon) \leq W_1(\mu, \nu)$ for all probability measures $\mu, \nu$) and $K^\star * U_\epsilon$ admits a density taking values in $[2/C, C/2]$ with respect to the measure $\lambda^\star * U_\epsilon$ (which is absolutely continuous with respect to the Lebesgue measure), so that it belongs to $\Omega_\omega^C$ (up to enlarging $\Lambda$). Moreover, $K^\star * U_\epsilon \longrightarrow K^\star$ in $\Omega_\omega^C$ as $\epsilon \longrightarrow 0$. Therefore, it remains to show that for all $\epsilon > 0$, the kernel $K^\star * U_\epsilon$ can be approximated by kernels in $\Omega_\omega^C$ with finite support. Equivalently, assume that $\lambda^\star$ is absolutely continuous with respect to the Lebesgue measure and construct a sequence approximating $K^\star$.

Let $(r_t)_{t \geq 1}$ and $((A_i^t)_{1 \leq i \leq r_t})_{t \geq 1}$ be the sequences obtained by applying Lemma~\ref{lemma_discretisation_support} to $\lambda^\star$. For all $t \geq 1$ and $i \in \{1, \dots, r_t\}$, let $x_i^t$ be an element of $A_i^t$. For all $t \geq 1$, the elements of the vector $\Xfrak_t = (x^t_i)_{1 \leq i \leq r_t}$ are distinct because $(A_i^t)_{1 \leq i \leq r_t}$ is a partition of $\Supp(\lambda^\star)$. Let $(\eta_t)_{t \geq 1} \longrightarrow 0$ be a sequence of positive numbers. Let $\tilde{K}_t$ be the transition kernel from $\Lambda \cap (\eta_t \Zbb^d)$ to $\{x_i^t\}_{1 \leq i \leq r_t}$ defined, for all $ x \in \Lambda \cap (\eta_t \Zbb^d)$ and all $ i \in \{1, \dots, r_t\}$, by
\begin{equation*}
\tilde{K}_t(x, x_i^t) = K^\star(x, A_i^t)\eqsp.
\end{equation*}
By the Lemma~\ref{lemma_discretisation_support} and assumption~H\ref{hyp_transitionKernel},  $\tilde{K}_t(x,x_i^t) \in [1/(C r_t),C/r_t]$ for all $x$ and $i$. Moreover, for all $x,x' \in \Lambda \cap (\eta_t \Zbb^d)$,
\begin{align*}
W_1(\tilde{K}_t(x, \cdot), \tilde{K}_t(x', \cdot))
	&\leq W_1(K^\star(x,\cdot), K^\star(x',\cdot)) + 2 \underset{1 \leq i \leq r_t}{\sup} \text{diam}(A_i^t) \\
	&\leq \frac{\omega(|x-x'|)}{2} + 2\frac{D_t}{\eta_t} |x-x'|\eqsp, \\
	&\leq \omega(|x-x'|)\eqsp,
\end{align*}
by choosing $\eta_t \geq 4 D_t / \inf_{u \in (0, \text{diam}(\Lambda)]} \omega(u)/u$, which is finite since $\omega$ is concave, nondecreasing and not equal to zero, so that there exists an extension $K_t \in \Omega_{\omega}^C$ of $\tilde{K}_t$ such that the support of $K_t(x,\cdot)$ is $\{x_i^t\}_{1 \leq i \leq r_t}$ for all $x \in \Lambda$. 

For all $i$, $j$, define $Q_t(i,j) = K_t(x^t_i, x^t_j)$. All kernels considered here ($K^\star$, $\tilde{K}_t$, $K_t$ and $K_{\Xfrak_t,Q_t}$) are kernels on the compact set $\Supp(\lambda^\star)$. Therefore, we only need to show that $K_{\Xfrak_t,Q_t} \longrightarrow K$ in the subset $\tilde{\Omega}_\omega^C$ of kernels on $\Supp(\lambda^\star)$ in $\Omega_\omega^C$ to show that it is an approximating sequence, that is
\begin{equation}
\label{eq_approx_cvg_unif_sur_support}
\sup_{x \in \Supp(\lambda^\star)} W_1(K_{\Xfrak_t,Q_t}(x,\cdot), K^\star(x,\cdot)) \underset{t \rightarrow +\infty}{\longrightarrow} 0\eqsp.
\end{equation}
For all $x \in \Supp(\lambda^\star)$, let $X(x)$ (resp. $\Xfrak(x)$) be one of the elements of $\Lambda \cap (\eta_t \Zbb^d)$ (resp. $\{x_i^t\}_{1 \leq i \leq r_t}$) closest to $x$. Then $\sup_{x \in \Supp(\lambda^\star)} |x-\Xfrak(x)| \leq D_t$ and $\sup_{x \in \Supp(\lambda^\star)} |x-X(x)| \leq \eta_t$ (with the supremum norm on $\R^d$) and for all $x \in \Supp(\lambda^\star)$,
\begin{align}
\nonumber
W_1(K_{\Xfrak_t,Q_t}(x, \cdot),  K^\star(x, \cdot)) \leq &W_1(K_{\Xfrak_t,Q_t}(x, \cdot), K_{\Xfrak_t,Q_t}(\Xfrak(x), \cdot)) \\
	\label{eq_approx_suite_2}
			&\hspace{.5cm}+ W_1(K_{\Xfrak_t,Q_t}(\Xfrak(x), \cdot), K_t(\Xfrak(x), \cdot)) \\
	\nonumber
			&\hspace{.5cm}+ W_1(K_t(\Xfrak(x), \cdot), K_t(X(\Xfrak(x)), \cdot)) \\
	\label{eq_approx_suite_4}
			&\hspace{.5cm}+ W_1(K_t(X(\Xfrak(x)), \cdot), K^\star(X(\Xfrak(x)), \cdot)) \\
	\nonumber
			&\hspace{.5cm}+ W_1(K^\star(X(\Xfrak(x)), \cdot), K^\star(x, \cdot))\eqsp.
\end{align}
By definition of the kernels, \eqref{eq_approx_suite_2} and  \eqref{eq_approx_suite_4} are equal to 0. Thus, the regularity assumptions on the kernels ensure that for all $x \in \Supp(\lambda^\star)$,
\begin{align*}
W_1(K_{\Xfrak_t,Q_t}(x, \cdot), K^\star(x, \cdot))
	\leq \omega(D_t)
			+ \omega(\eta_t)
			+ \omega(D_t + \eta_t) / 2\eqsp,
\end{align*}
which proves Equation~\eqref{eq_approx_cvg_unif_sur_support}.

\subsection{Proof of Proposition~\ref{prop_oracle_simplifie}}
\label{sec_full_oracle_inequality}

This section first states Theorem 8 of~\cite{lehericy2018misspe} and its assumptions. It is then proved that the assumptions are satisfied and that Proposition~\ref{prop_oracle_simplifie} is deduced from this theorem. Let $\lambda_b$ be the probability measure on $\R^d$ which has the density $b / \|b\|_1$ with respect to the Lebesgue measure. When necessary, the process $(Y_t)_{t \geq 1}$ is extended to a process $(Y_t)_{t \in \Zbb}$ by stationarity.
In this section only, for all integers $a \leq b$, write $Y_a^b$ instead of $(Y_a, \dots, Y_b)$.

\begin{description}
\item[\textbf{[A$\star$forgetting]}] There exists two constants $C_\star > 0$ and $\rho_\star \in (0,1)$ such that for all $i \in \Zbb$, for all $k,k' \in \Nbb^*$ and for all $y_{i-(k \vee k')}^i \in (\R^d)^{(k \vee k')+1}$,
\begin{equation*}
\left| 
	\log\left(\frac{\rmd \po_{Y_i | Y_{i-k}^{i-1}, K^\star, \gamma^\star}}{\rmd \lambda_b}(y_i | y_{i-k}^{i-1})\right)
	- \log\left(\frac{\rmd \po_{Y_i | Y_{i-k'}^{i-1}, K^\star, \gamma^\star}}{\rmd \lambda_b}(y_i | y_{i-k'}^{i-1})\right)
\right|
	\leq C_\star \rho_\star^{k \wedge k' - 1}\eqsp.
\end{equation*}
\end{description}
Let $(\Omega, \Fcal, P)$ be a measured space and $\Acal \subset \Fcal$ and $\Bcal \subset \Fcal$ be two sigma-fields. Then, the $\rho$-mixing coefficient between $\Acal$ and $\Bcal$ is
\begin{equation*}
\rho_\text{mix}(\Acal, \Bcal) = \underset{g \in \Lbf^2(\Omega,\Bcal,P)}{\sup_{f \in \Lbf^2(\Omega,\Acal,P)}} | \text{Corr}(f,g)|\eqsp.
\end{equation*}
The $\rho$-mixing coefficient of $(Y_t)_{t \in \Zbb}$ is 
\begin{equation*}
\rho_\text{mix}(n) = \rho_\text{mix}(\sigma(Y_i, i \geq n), \sigma(Y_i, i \leq 0))\eqsp.
\end{equation*}

\begin{description}
\item[\textbf{[A$\star$mixing]}] There exists two constants $c_\star > 0$ and $n_\star \in \Nbb^*$ such that for all $n \geq n_\star$, $\rho_\text{mix}(n) \leq 4 e^{- c_\star n}$.

\item[\textbf{[A$\star$tail]}] There exists a constant $B^\star \geq 1$ such that for all $i \in \Zbb$, all $k \in \Nbb$ and all $v \geq e$,
\begin{equation*}
\po\left( \frac{\rmd \po_{Y_i | Y_{i-k}^{i-1}, K^\star, \gamma^\star}}{\rmd \lambda_b}(Y_i | Y_{i-k}^{i-1}) \geq v^{B^\star} \right) \leq \frac{1}{v}\eqsp.
\end{equation*}

\end{description}
\cite{lehericy2018misspe} considers models  written $T_{r,D}$ in the following (instead of $S_{K,M,n}$ in \cite{lehericy2018misspe}). These models are sets of hidden Markov model parameters (not translation hidden Markov models), that is of vectors of the form $(r, \pi, Q, g)$ where $r$ is the number of values the Markov chain can take, $\pi$ is the initial distribution of the Markov chain, $Q$ is its transition matrix and $g = (g_z)_{z = 1, \dots, r}$ the vector of its emission densities, that is a vector of probability densities on $\R^d$ with respect to the Lebesgue measure. Let $(m_{r,D})_{r \geq 1, D \geq 1}$ be a sequence of nonnegative integers. For all $n \geq 1$, let $\sigma_-(n) \in (0, e^{-1}]$ and let $\Pfrak_n$ be a subset of $\{ (r,D) \in (\N^*)^2 : r \leq 1 / (2\sigma_-(n)) \text{ and } m_{r,D} \leq 2n\}$. This set lists the indices of the models among which the final model is selected.
Let $\Tbf_n = \bigcup_{(r,D) \in \Pfrak_n} T_{r,D}$ be the set of all model parameters considered when $n$ observations are available.

\begin{description}
\item[\textbf{[Aergodic]}] For all $(r,\pi,Q,-) \in \Tbf_n$,
\begin{equation*}
\inf_{x,x' = 1, \dots, r} Q(x,x') \geq \sigma_-(n) \quad \text{ and } \quad
\inf_{x = 1, \dots, r} \pi(x) \geq \sigma_-(n)\eqsp.
\end{equation*}
\end{description}

\begin{description}
\item[\textbf{[Atail]}]
There exists a constant $B(n) \geq 1$ such that for all $u \geq 1$,
\begin{equation*}
\po^\star \left( \sup_{(r,-,-,g) \in \Tbf_n} \left| \log \sum_{z=1}^r g_z(Y_1) \right| \geq B(n) u \right) \leq e^{-u}\eqsp.
\end{equation*}
\end{description}
Finally, the assumptions \textbf{[Aentropy]} and \textbf{[Agrowth]} of \cite{lehericy2018misspe} are replaced by the following more general assumption, which allows to improve the penalty (the original assumptions induce a penalty proportional to $r\dim_D + rd + r^2$ instead of $\dim_D + rd + r^2$). Let $N\left(B, d, \epsilon \right)$ be the smallest number of brackets of size $\epsilon$ for the distance $d$ needed to cover the set of functions $B$.

\begin{description}
\item[\textbf{[Aentropy']}]
There exist a mapping $(r, D, n, A) \longmapsto C_\text{aux}(r, D, n, A) \geq 1$, a sequence of nonnegative integers $(m_{r,D})_{r \geq 1, D \geq 1}$ and a family of sets $(\Scal_{n,A})_{n \geq 1, A \geq 0} \subset \R^d$ such that for all $n \geq 1$ and $A \geq 0$, $\po^\star(Y_1 \notin \Scal_{n,A}) \leq \exp(-2A / B(n))$ where $B(n)$ is as in \textbf{[Atail]}, for all $y \in \Scal_{n,A}$,
\begin{equation*}
\underset{(r',-,-,g') \in \Tbf_n}{\sup} \left| \log \sum_{z=1}^{r'} g'_z(y) \right| \leq A
\end{equation*}
and for all $r \geq 1$, $D \geq 1$, $n \geq 1$, $A \geq B(n)$ and $\delta \in (0,1)$,
\begin{multline}
\label{eq_bracketing_entropy}
N\left(\left\{ \left(y \mapsto g_z(y) \one_{y \in \Scal_{n,A}} \right)_{z = 1, \dots, r} \right\}_{(r,-,-,g) \in T_{r,D}}, d_\infty, \delta \right)  \\
	\leq \max \left( \frac{C_\text{aux}(r, D, n, A)}{\delta} , 1\right)^{m_{r,D}}\eqsp,
\end{multline}
where $d_\infty$ is the distance associated with the supremum norm on $(\Lbf^\infty(\mathcal{Y}))^r$.
Moreover, there exist an integer $n_\text{growth}$ and a constant $c_\text{growth} > 0$ such that for all $n \geq n_\text{growth}$,
\begin{equation*}
\sup_{(r, D) \in \Pfrak_n} \log C_\text{aux}(r,D,n, B(n) \log n) \leq c_\text{growth} (\log n)^2 \log \log n\eqsp.
\end{equation*}
\end{description}
Note that choosing $S_{n,A} = \{ y \in \R^d : \sup_{(r',-,-,g') \in \Tbf_n} | \log \sum_{z=1}^{r'} g'_z(y) | \leq A \}$ gives the original formulation of~\cite{lehericy2018misspe}. Write $\po_{r,\pi,Q,g}$ the distribution of a hidden Markov model with parameter $(r,\pi,Q,g)$. Lemma 4 and 5 of~\cite{lehericy2018misspe} show that for all $r$, $D$ and for all $(r,\pi,Q,g) \in T_{r,D}$, the limit
$\Kbf(\po_{K^\star, \gamma^\star} \| \po_{r,Q,g}) = \lim_m m^{-1} KL( \po_{Y_1^m | K^\star, \gamma^\star} \| \po_{Y_1^m | r,\pi,Q,g} )$
exists, is finite and does not depend on $\pi$. This quantity coincides with the one defined in Lemma~\ref{lem_KL_equivalent_VT} when the hidden Markov model with parameter $(r,\pi,Q,g)$ is a translation hidden Markov model with transition kernel in $\Omega_\omega^C$ and emission density in $\Gamma$. Define the loglikelihood of a hidden Markov model with parameter $(r,\pi,Q,g)$ by
\begin{equation*}
\ell_{n}^\text{HMM}(r,\pi,Q,g) = \log \left(\sum_{z_1, \ldots, z_n\in \{1,\ldots,r\}}\!\!\!\!\!\!\!\!\!\!\pi(z_1) g_{z_1}(Y_1)\prod_{t=2}^{n}Q(z_{t-1},z_{t}) g_{z_t}(Y_t)
\right).
\end{equation*}
Theorem 8 of~\cite{lehericy2018misspe} may now be stated with a noteworthy modification: not all possible number of states and model indices are considered during the model selection step~\eqref{eq_model_selection_thLuc}, but only the ones in $\Pfrak_n$. This has no consequence on the proof.

\begin{theo}
\label{TH_ORACLE_SIMPLIFIE}
Assume that \textbf{[A$\star$forgetting]}, \textbf{[A$\star$mixing]}, \textbf{[A$\star$tail]},
\textbf{[Aergodic]}, \textbf{[Atail]} and \textbf{[Aentropy']} hold.
Assume that ${\sigma_-(n) = C_\sigma (\log n)^{-1}}$ and ${B(n) = C_B \log n}$ for some constants $C_\sigma \geq 0$ and $C_B \geq 2$. Let $\alpha \geq 0$.  For all $r$ and $D$, let
\begin{equation*}
{(r, \hat\pi_{r,D,n}, \hat{Q}_{r,D,n}, \hat{g}_{r,D,n})}
	\in \underset{(r,\pi,Q,g) \in T_{r,D}}{\argmax} \; \frac{1}{n} \ell_{n}^\text{HMM}(r,\pi,Q,g)\eqsp,
\end{equation*}
\begin{equation}
\label{eq_model_selection_thLuc}
(\hat{r}_n, \hat{D}_n) \in \underset{(r, D) \in \Pfrak_n}{\argmax} \left( \frac{1}{n} \ell_{n}^\text{HMM}(r, \hat\pi_{r,D,n}, \hat{Q}_{r,D,n}, \hat{g}_{r,D,n}) - \pen(n,r,D) \right)\eqsp,
\end{equation}
for some function $\pen$, and let
\begin{equation*}
(\hat{r}_n, \hat{\pi}_n, \hat{Q}_n, \hat{g}_n) = (\hat{r}_n, \hat\pi_{\hat{r}_n,\hat{D}_n,n}, \hat{Q}_{\hat{r}_n,\hat{D}_n,n}, \hat{g}_{\hat{r}_n,\hat{D}_n,n})
\end{equation*}
be the nonparametric maximum likelihood estimator.
Then, there exist constants $A$, $C_\pen$ and $n_0$ depending only on $\alpha$, $C_\sigma$, $C_B$, $n_*$, $c_*$ and $c_\text{growth}$ such that for all
\begin{equation*}
n \geq n_\text{growth}
		\vee n_0
		\vee \exp\left(C_\sigma \left( (1 + C_*) \vee \frac{2-\rho_*}{1-\rho_*} \vee e^2 \right) \right) 
		\vee \exp\left(\frac{B^*}{C_B}\right)
		\vee \exp \sqrt{\frac{C_\sigma}{2}(n_*+1)}\eqsp,
\end{equation*}
all $t \geq 1$,  all $\eta \leq 1$, with probability at least $1 - e^{-t} - 2 n^{-\alpha}$,
\begin{multline*}
\Kbf( \po_{K^\star, \gamma^\star} \| \po_{\hat{r}_n, \hat{Q}_n, \hat{g}_n})
	\leq (1+\eta) \underset{(r,D) \in \Pfrak_n}{\inf} \left\{ \inf_{(r,\pi,Q,g) \in T_{r,D}}\Kbf( \po_{K^\star, \gamma^\star} \| \po_{r,Q,g} )
	 + 2 \pen(n,r,D) \right\} \\
		+ \frac{A}{\eta} t \frac{(\log n)^8}{n}
\end{multline*}
as soon as
\begin{equation*}
\pen(n,r,D) \geq
	\frac{C_\pen}{\eta} ( m_{r,D} + r^2 - 1) \frac{(\log n)^{14} \log \log n}{n}\eqsp.
\end{equation*}
\end{theo}

Let us now check the assumptions. \textbf{[A$\star$mixing]} and \textbf{[A$\star$forgetting]} follow from Lemma 1 of~\cite{lehericy2018misspe} and from H\ref{hyp_transitionKernel} with $\rho_\star = 1 - 4/C^2$, $C_\star = C^2/4$, $n_\star = 1$ and $c_\star = -\log(1 - 2/C)/2$, where $C$ is the constant from H\ref{hyp_transitionKernel}. \textbf{[A$\star$tail]} follows from assumption H\ref{hyp_densiteEmission} with $B^\star = \max(1, \log \|b\|_1)$: by definition of~$\lambda_b$ and $b$, for all $i \in \Zbb$, $k \in \N$, $y_{i-k}^i \in (\R^d)^{k+1}$ and $v \geq e$,
\begin{align*}
\frac{\rmd \po_{Y_i | Y_{i-k}^{i-1}, K^\star, \gamma^\star}}{\rmd \lambda_b}(y_i | y_{i-k}^{i-1})
	= \frac{\int \gamma^\star(y_i - x) \rmd\po_{X_i | Y_{i-k}^{i-1}, K^\star, \gamma^\star}(x | y_{i-k}^i)}{b(y_i) / \|b\|_1}
	\leq \|b\|_1 \leq v^{B^\star} \eqsp.
\end{align*}
For each $r \geq 1$ and $D \geq 1$, let $m_{r,D} = \dim_D + rd$. For each $n \geq 1$, let $\sigma_-(n) = (2\log n)^{-1}$ and $\Pfrak_n = \{ (r,D) : r \leq \log n \text{ and } \dim_D \leq n \}$. For $n$ large enough, $\Pfrak_n$ is indeed a subset of $\{ (r,D) \in (\N^*)^2 : r \leq 1 / (2\sigma_-(n)) \text{ and } m_{r,D} \leq 2n\}$. For each $r \geq 1$ and $D \geq 1$, the model $T_{r,D}$ is the set of translation hidden Markov model parameters in $S_{r,D}$ seen as hidden Markov model parameters (with the dominating measure $\lambda_b$ on $\R^d$ instead of the Lebesgue measure):
\begin{equation*}
T_{r,D} = \left\{ \left(r,\pi_Q,Q,\left(y \longmapsto \frac{\gamma(y - x_r)}{b(y) / \|b\|_1}\right)_{z = 1, \dots, r} \right) : ((x_z)_{z=1, \dots, r}, Q, \gamma) \in S_{r,D}, \pi_Q Q = \pi_Q \right\} \eqsp.
\end{equation*}
By definition of $S_{r,D}$, for all $(r,\pi,Q,-) \in T_{r,D}$ and $x,x' \in \{1, \dots, r\}$, $Q(x,x') \geq (Cr)^{-1}$ and $\pi(x) \geq (Cr)^{-1}$. Thus, for all $(r,\pi,Q,-) \in \Tbf_n$, $Q(x,x') \geq (C \log n)^{-1} \geq \sigma_-(n)$ since $C \geq 2$. The same holds for $\pi$, so that \textbf{[Aergodic]} is satisfied.

By H\ref{hyp_densiteEmission}, for all $n \geq 1$ and $y \in \R^d$, $\sup_{(r,-,-,g) \in \Tbf_n} \sum_{z=1}^r g_z(y) \leq \|b\|_1 \log n$, and by H\ref{hyp_tails},
$$
\inf_{(r,-,-,g) \in \Tbf_n} \sum_{z=1}^r g_z(y)
	\geq \|b\|_1 m(y)/b(y)\eqsp,
$$
so that by Markov's inequality, for all $t > 0$, with $\epsilon$ as in H\ref{hyp_tails},
\begin{align*}
\po_{K^\star, \gamma^\star} \left[ 
	\left(\inf_{(r,-,-,g) \in \Tbf_n} \sum_{z=1}^r g_z(y)\right)^{-\epsilon}
		\geq t
\right]
	\leq \|b\|_1^{-\epsilon} \frac{\E_{K^\star, \gamma^\star} [(b(Y_1)/m(Y_1))^{\epsilon}]}{t}\eqsp,
\end{align*}
so that there exists a constant $C_{H\ref{hyp_tails}} > 0$ such that
\begin{align*}
\po_{K^\star, \gamma^\star} \left[ 
	\inf_{(r,-,-,g) \in \Tbf_n} \log \sum_{z=1}^r g_z(y)
		\leq - \frac{1}{\epsilon} u
\right]
	\leq C_{H\ref{hyp_tails}} e^{-u}\eqsp.
\end{align*}
Thus, there exists $n_\text{tail}$ such that \textbf{[Atail]} holds for any $n \geq n_\text{tail}$ and for any $B(n) \geq \max(2/\epsilon, \log(\|b\|_1 \log n))$. Choose $B(n) = \log n$.

Finally, \textbf{[Aentropy']} is implied by the following assumption, which follows from H\ref{hyp_densiteEmission} and H\ref{hyp_modelesParam} with $c(r,D,A) = c(D,A) + C_\Gamma$.
\begin{description}
\item[\textbf{[Aentropy'']}] There exists a mapping $(r,D,A) \in \N^* \times \N^* \times \R_+ \longmapsto c(r,D,A)$ and a constant $c'$ such that $\log c(r,D,A) \leq c' ( \log m_{r,D} + A )$. There exists a sequence $(\Theta_D)_{D \geq 1}$ of sets such that for all $D \geq 1$, $\Theta_D \subset [-1,1]^{\dim_D}$ and there exists a surjective mapping $\theta \in \Theta_D \longmapsto \gamma^\theta \in G_D$. For all $r \geq 1$, $D \geq 1$, $A \geq 0$ and $y \in \R^d$ such that $\log(b(y)/m(y)) \leq A$, the mapping
$
(x, \theta) \in \Lambda^r \times \Theta_D \longmapsto (\gamma^\theta(y - x_z) / b(y))_{z \in \{1, \dots, r\}}
$
is $c(r,D,A)$-Lipschitz (when $\Lambda$ and $\Theta_D$ are endowed with the supremum norm).
\end{description}
Let us see how this implies \textbf{[Aentropy']}.
Let $\Scal_{n,A} = \{ y \in \R^d : \log(b(y)/m(y)) \leq A \}$. By H\ref{hyp_tails} and Markov's inequality, $\po^\star(Y_1 \in \Scal_{n,A}) \leq \exp(- A \epsilon / 2)$ for $A$ large enough. Moreover, for all $A \geq \log(\|b\|_1 \log n)$ and $y \in \Scal_{n,A}$,
\begin{equation*}
\underset{(r',-,-,g') \in \Tbf_n}{\sup} \left| \log \sum_{z=1}^{r'} g'_z(y) \right|
	\leq \max\left( \log \frac{b(y)}{\|b\|_1 m(y)}, \log (\|b\|_1 \log n) \right) \leq A.
\end{equation*}
A bracket covering of size $\delta$ of $[-1,1]^{rd} \times [-1,1]^{\dim_D}$ gives a bracket covering of size $\delta L$ of $\Lambda^r \times \Theta_D$, which in turn gives bracket covering of size $c(r,D,A) \delta L \|b\|_1$ of the set
\begin{equation*}
\left\{ \left(y \longmapsto \|b\|_1 \frac{\gamma(y-x_z)}{b(y)} \one_{y \in \Scal_{n,A}} \right)_{z = 1, \dots, r} : x \in \Lambda^r, \gamma \in G_D \right\}.
\end{equation*}
Since there exists a bracket covering of size $\delta$ of $[-1,1]$ with cardinality at most $\max(2/\delta, 1)$, Equation~(\ref{eq_bracketing_entropy}) of \textbf{[Aentropy']} holds with $C_\text{aux}(r,D,n,A) = 2 c(r,D,A) L \|b\|_1$.
Finally, since $\sup_{(r,D) \in \Pfrak_n} \log c(r,D,A) \leq c' (\log n + A)$, the last part of \textbf{[Aentropy']} holds.

Thus, Theorem~\ref{TH_ORACLE_SIMPLIFIE} holds and ensures that there exists $n_0$, $C_\pen$ and $A$ such that if $\pen(n,r,D) \geq C_\pen(m_{r,D}+r^2-1) (\log n)^{14}/n$, then for all $n \geq n_0$ and $t \geq 1$, with probability at least $1 - e^{-t} - 2n^{-2}$,
\begin{multline*}
\Kbf( \po_{K^\star, \gamma^\star} \| \po_{\hat{\Xfrak}_n, \hat{Q}_n, \hat{\gamma}_n})
	\leq 2 \underset{(r,D) \in \Pfrak_n}{\inf} \left\{ \inf_{(\Xfrak,Q,\gamma) \in S_{r,D}}\Kbf( \po_{K^\star, \gamma^\star} \| \po_{\Xfrak,Q,\gamma} )
	 + 2 \pen(n,r,D) \right\} \\
		+ A t \frac{(\log n)^8}{n}
\end{multline*}
and Proposition~\ref{prop_oracle_simplifie} follows by taking $t = 2 \log n$ and recalling that $m_{r,D} = \dim_D + rd$ and $\Pfrak_n = \{(r,D) : r \leq \log n \text{ and } \dim_D \leq n \}$.

\section{Additional Simulations based on Least Squares for Characteristic Functions}
\label{sec:simu:appendix}
In this section, the empirical least squares criterion $M_n(R)$ introduced in Section~\ref{sec:least:squares} is approximated to obtain a practical estimate of $R$ using the same model as in Section~\ref{sec:simu}. 
The estimate $\widehat {\Phi}_{n}$ of the characteristic function of the observations $(Y_1,Y_2)$ is given for all $(t_{1},t_{2})\in\mathbb{R}^2$ by
\[
\widehat {\Phi}_{n}(t_1,t_2) = \frac{1}{n}\sum_{j=1}^{n-1} \mathrm{e}^{it_1 Y_j + it_2 Y_{j+1}}\eqsp.
\]
The function $w$ is set as the probability density function of a Gaussian random variable with standard deviation $\sigma =3$ and $M_n$ is estimated by the Monte Carlo estimate:
\[
\widehat{M}_n(R) = \frac{1}{N}\sum_{\ell=1}^N \left\vert   \widehat {\Phi}_{n}(U_1^\ell,U_2^\ell)\Phi_{R}(U_1^\ell;0)\Phi_{R}( 0; U_2^\ell)  -
  \Phi_{R}(U_1^\ell,U_2^\ell)\widehat {\Phi}_{n}(U_1^\ell;0)\widehat {\Phi}_{n}( 0; U_2^\ell)
  \right\vert^{2}\eqsp,
\]
where $(U_1^\ell,U_2^\ell)_{1\leq \ell \leq N}$ are independent and identically distributed with distribution $w$. In the following experiments, $N$ is set to $5000$. This estimated criterion is minimized over the set $\mathcal{D}_r$ of piecewise constant probability densities on $(-1,1)\times(-1,1)$ with $r^2$ uniformly spaced cells:
\[
\mathcal{D}_r = \Big\{R: \mathbb{R}^2\to \mathbb{R}_+\eqsp;\eqsp R = \sum_{i,j=1}^r \alpha_{i,j} \mathds{1}_{(x_i,x_{i+1})\times(x_j,x_{j+1})}\Big\}\eqsp,
\]
where for all $1\leq i,j\leq r$, $x_i = -1 + 2(i-1)/r$, $\alpha_{i,j}\geq 0$ and $\sum_{i,j=1}^r \alpha_{i,j} = r^{-2}$. In this setting where the support of the law of $(X_1,X_2)$ is compact and known, the up to translation indeterminacy is ruled out. The optimization is performed using the Covariance Matrix Adaptation Evolutionary Strategy \cite{igel:hansen:roth:2007} which optimizes iteratively all parameters using $(\mu,\lambda)$-selection. At each iteration, the best offsprings of the current parameter estimate are combined to form the population of the following iteration and the other offsprings are discarded.  

The performance of the least squares approach is assessed by comparing the estimated probability that $(X_1,X_2)$ lies in each cell $(x_i,x_{i+1})\times(x_j,x_{j+1})$, $1\leq i,j\leq r$, which is $\widehat\alpha^n_{i,j} r^2$ and the benchmark estimation  $\widetilde\alpha^{\mathsf{n,emp}}_{i,j}$ that would be computed if the sequence $(X_k)_{1\leq k \leq n}$ were observed: $\widetilde p^{\mathsf{n,emp}}_{i,j}  = n^{-1}\sum_{k=1}^{n-1}\mathds{1}_{(x_i,x_{i+1})\times(x_j,x_{j+1})}(X_k,X_{k+1})$. The results are displayed in Figure~\ref{fig:leastsquares:L1} over $10$ independent runs, when the order $r$ is in $\{10,20,30\}$, with CMA-ES initialized at a random point, and a maximum number of evaluations of  $\widehat{M}_n(R)$ set to 75000. Each estimate is obtained with a sequence of $n=100000$ observations and the $\mathrm{L}_1$ score is 
\begin{equation}
\label{eq:L1scores}
\varepsilon^r_{1,n} = \frac{1}{r^2}\sum_{i,j=1}^{r}\left|r^2\widehat\alpha^n_{i,j}-\widetilde p^{\mathsf{n,emp}}_{i,j}\right|\eqsp.
\end{equation} 
The associated estimated probabilities for the distribution of $X_1$ are displayed in Figure~\ref{fig:leastsquares:X1} with their confidence regions.

\begin{figure}
\includegraphics[width=.9\textwidth]{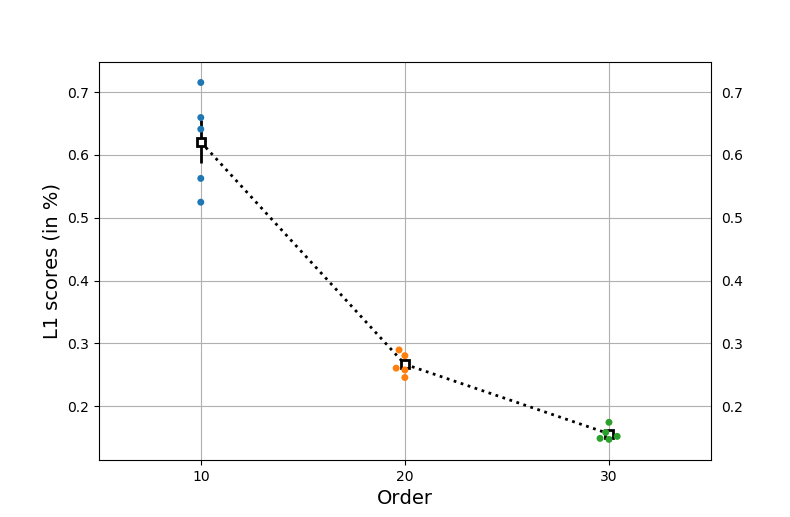}
\caption{$\mathrm{L}_1$ scores computed according to \eqref{eq:L1scores}. Each dot is an estimated value with the least squares approach. For each value of $r$, the mean value (squares) over all runs as long as the empirical standard deviation (bars) are displayed.}
\label{fig:leastsquares:L1}
\end{figure}

\begin{figure}
\includegraphics[width=.83\textwidth]{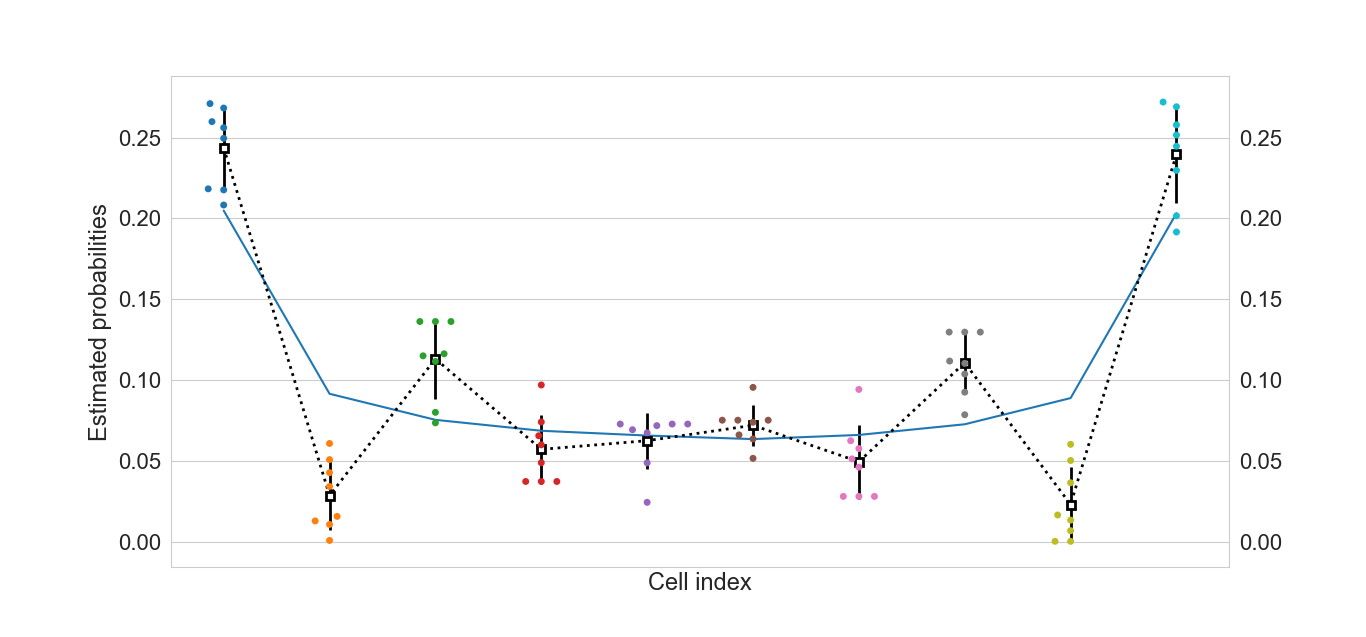}

\includegraphics[width=.83\textwidth]{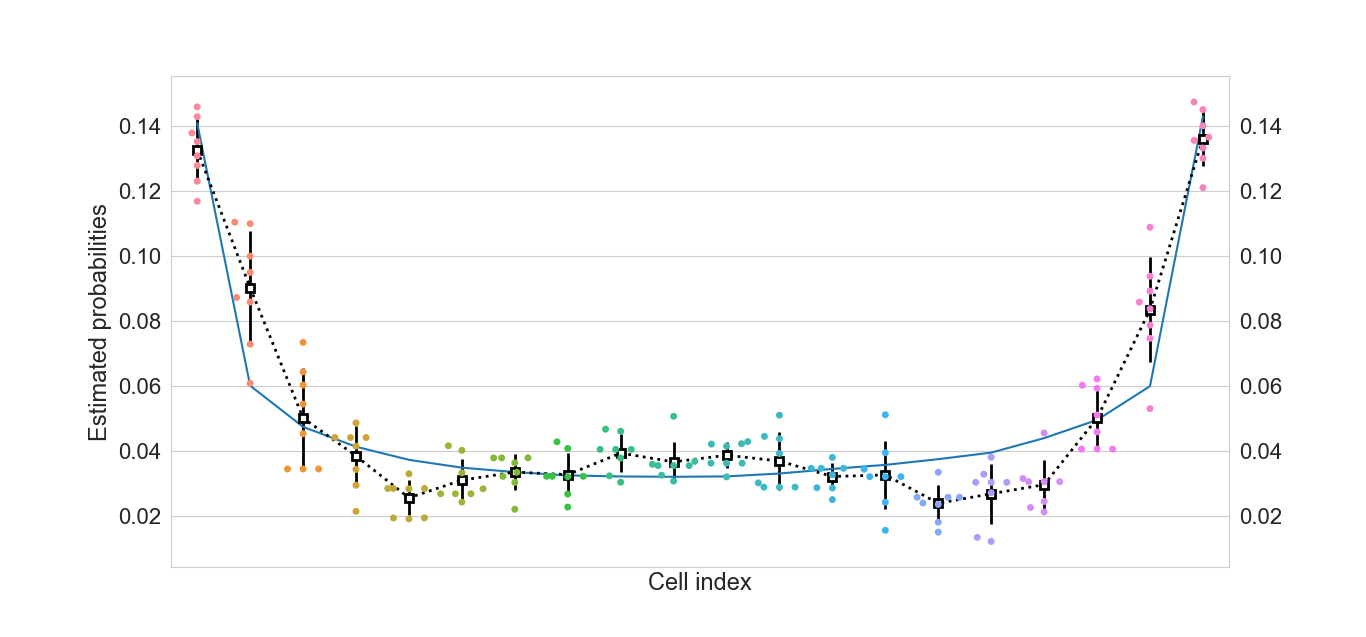}

\includegraphics[width=.83\textwidth]{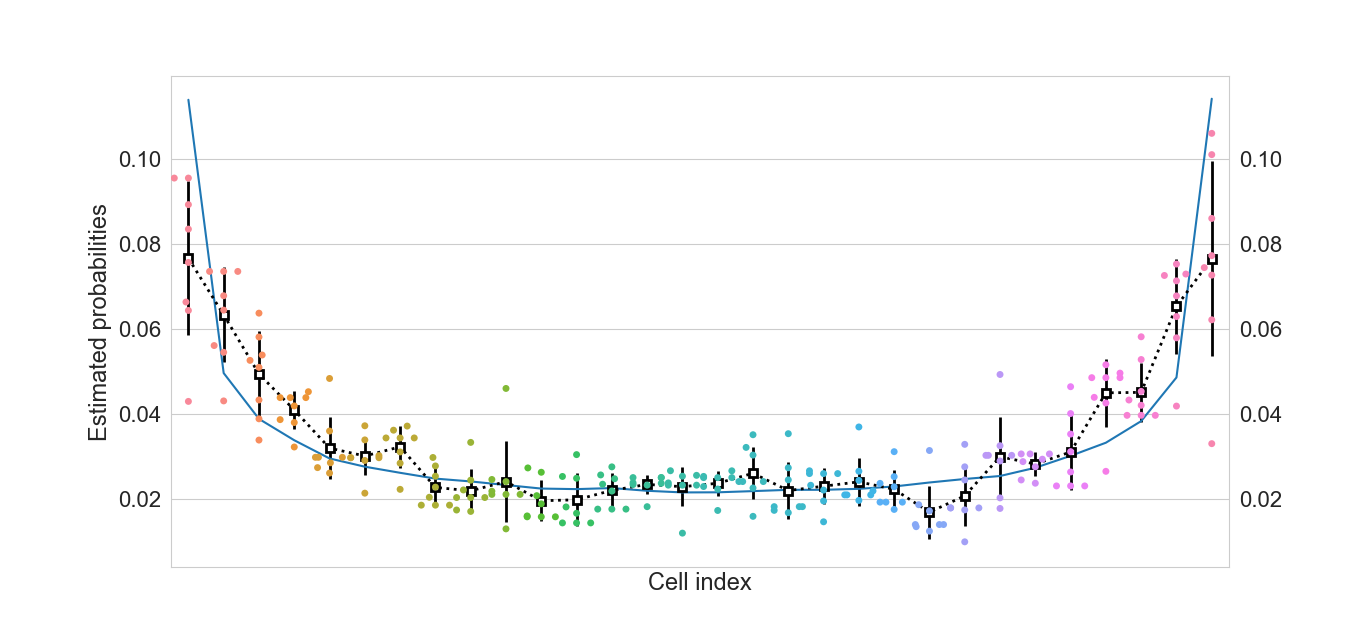}

\caption{Estimated probabilities associated with the marginal distribution of $X_1$ for $r=10$ (top), $r = 20$ (middle) and $r = 30$ (bottom). The blue line is the empirical estimate when the sequence $(X_k)_{1\leq k \leq n}$ is observed (mean estimate over the 10 Monte Carlo runs). Each dot is an estimated value with the least squares approach. For each value of $r$, the mean value (squares) over all runs as long as the empirical standard deviation (bars) are displayed.}
\label{fig:leastsquares:X1}
\end{figure}





\end{document}